\documentclass[12pt]{amsart}
\usepackage{amssymb}
\usepackage{amsthm}
\usepackage{bm}
\usepackage{latexsym}
\usepackage{amsmath}
\usepackage{eufrak}
\usepackage{mathrsfs}
\usepackage{amscd}
\usepackage[all]{xy}
\usepackage[usenames]{color}
\usepackage[dvips]{graphicx}

\theoremstyle{plain}
\newtheorem{thm}{Theorem}[section]
\newtheorem{cor}[thm]{Corollary}
\newtheorem{conj}[thm]{Conjecture}
\newtheorem{lem}[thm]{Lemma}
\newtheorem{prop}[thm]{Proposition}

\theoremstyle{definition}
\newtheorem{defi}{Definition}[section]
\newtheorem{rmk}{Remark}[section]
\newtheorem{example}{Example}[section]

  \DeclareMathOperator{\spec}{Spec}

\def\O{{\mathcal{O}}}

\def\L{{\mathcal{L}}}
\def\F{{\mathbb{F}}}
\def\FF{{\mathcal{F}}}
\def\ZZ{{\mathcal{Z}}}
\def\K{{\mathbb{K}}}
\def\R{{\mathbb{R}}}

\def\A{{\mathcal{A}}}
\def\BB{{\mathcal{B}}}

\def\BB{{\mathscr{B}}}

\def\S{{\mathcal{S}}}

\def\Z{{\mathbb Z}}
\def\Q{{\mathbb Q}}

\def\PP{{\mathscr P}}
\def\P{{\mathbb P}}
\def\C{{\mathbb C}}

\def\coh{{\text{H}}}
\def\CC{{\mathcal C}}

\def\nA{{\bar{\A}}}

\def\spec{\textrm{Spec}}
\def\Pic{\textrm{Pic}}

\def\deg{\textrm{deg}}

\begin{document}

\title[Construcci\'on]{Arrangements of curves and algebraic surfaces}
\author{\textrm{Giancarlo Urz\'ua}}

\email{gian@umich.edu}

\maketitle

\begin{center} {\small \textit{Dedicated to F. Hirzebruch on the occasion of his 80th birthday}} \end{center}

\begin{abstract}
We prove a strong relation between Chern and log Chern invariants of algebraic surfaces. For a given arrangement of curves, we find nonsingular
projective surfaces with Chern ratio arbitrarily close to the log Chern ratio of the log surface defined by the arrangement. Our method is based
on sequences of random $p$-th root covers, which exploit a certain large scale behavior of Dedekind sums and lengths of continued fractions. We
show that randomness is necessary for our asymptotic result, providing another instance of ``randomness implies optimal". As an application over
$\C$, we construct nonsingular simply connected projective surfaces of general type with large Chern ratio. In particular, we improve the
Persson-Peters-Xiao record for Chern ratios of such surfaces.
\end{abstract}

\section{Introduction.} \label{s1}

Let $X$ be a nonsingular projective surface over an algebraically closed field $\K$. The Chern numbers of $X$ are defined via Chern classes of
its sheaf of differentials as $$c_1^2(X):= c_1 \big({\Omega_X^1}^{*} \big)^2 \ \ \ \ \text{and} \ \ \ \ c_2(X):= c_2 \big({\Omega_X^1}^{*}
\big).$$ These numerical invariants are important from the point of view of classification, being the analogues of the genus of a nonsingular
projective curve.

On the other hand, let $(Y,\nA)$ be a nonsingular log surface over $\K$ (see \cite{IitakaAG82,Sakai1}). This means, $Y$ is a nonsingular
projective surface over $\K$, and $\nA$ is a simple normal crossing divisor in $Y$. We think of the pair $(Y,\nA)$ as the open surface $Y
\setminus \nA$ compactified by $\nA$. As in the projective case, one defines the log Chern numbers of $(Y,\nA)$ via Chern classes of its sheaf
of log differentials as $$ \bar{c}_1^2(Y,\nA):= c_1 \big({\Omega_Y^1(\log \nA)}^{*} \big)^2 \ \ \ \ \text{and} \ \ \ \ \bar{c}_2(Y,\nA):= c_2
\big({\Omega_Y^1(\log \nA)}^{*} \big).$$

Let $d\geq 3$ be an integer, and let $Z$ be a nonsingular projective surface over $\K$. An arrangement of $d$ curves $\A$ in $Z$ is a collection
of $d$ nonsingular projective curves $\{ C_1,\ldots,C_d \}$ such that $\bigcap_{i=1}^d C_i = \emptyset$. The pair $(Z,\A)$ uniquely defines a
nonsingular log surface $(Y,\bar{\A})$ by considering the minimal log resolution $\sigma: Y \rightarrow Z$ of $\A$, and defining $\nA:=
\sigma^*(\A)_{\text{red}}$.

The main result of this article is the following strong relation between Chern and log Chern invariants for algebraic surfaces.

\vspace{0.2 cm} \textbf{Theorem \ref{t71}.} Let $Z$ be a nonsingular projective surface over $\K$, and let $\A$ be a \textbf{simple crossing}
\textbf{divisible} arrangement of curves in $Z$ (see Definitions \ref{d51} and \ref{d52}). Consider the associated nonsingular log surface
$(Y,\nA)$, and assume $\bar{c}_2(Y,\nA) \neq 0$. Then, there exist nonsingular projective surfaces $X$ with {\large$\frac{c_1^2(X) }{c_2(X)}$}
arbitrarily close to {\large $\frac{\bar{c}_1^2(Y,\nA)}{\bar{c}_2(Y,\nA)}$}. \vspace{0.3cm}

Simple crossing divisible arrangements are quite abundant in nature. For example, any arrangement of pairwise-transversal nonsingular plane
curves satisfies this condition. Divisible arrangements are key in proving Theorem \ref{t71}, producing partitions of arbitrarily large prime
numbers $p$, which in turn assign multiplicities to the curves in the arrangement. Then, in order to produce nonsingular projective surfaces, we
consider sequences of $p$-th root covers branch along this weighted arrangement.

A central ingredient in the computation of invariants is the occurrence of Dedekind sums and lengths of continued fractions. In Section
\ref{s4}, we show that the effect of multiplicities in the invariants is encoded through these two arithmetic quantities. It turns out that not
all assignments of multiplicities give the asymptotic result of Theorem \ref{t71}. However, random assignments work for that purpose. We use a
certain large scale behavior of Dedekind sums and lengths of continued fractions, recently discovered by Girstmair \cite{Gi03,Gi06}, to prove
the existence of ``good" partitions. These partitions produce the surfaces $X$ in Theorem \ref{t71}. At the same time, we show that random
partitions are ``good" with probability tending to $1$ as $p$ becomes arbitrarily large (see proof of Theorem \ref{t71}).

An interesting phenomenon is that \textbf{random partitions} of prime numbers are indeed necessary for our result. This shows another instance
of the slogan ``randomness implies optimal", the latter meaning asymptotically close to log Chern ratios. We put this in evidence by using a
computer program which calculates the exact values of the Chern numbers involved (see Remark \ref{r71}). In this way, we introduce the notion of
\textbf{random surface} associated to an arrangement of curves (see Definition \ref{d71}).

As an application over the complex numbers, we show that, for certain arrangements of curves, random surfaces provide examples of simply
connected surfaces of general type with large $\frac{c_1^2(X)}{c_2(X)}$ (see Section \ref{s9}). For instance, we prove in Theorem \ref{t91} that
random surfaces associated to line arrangements in $\P^2(\C)$ are simply connected, and have large Chern ratio in general, but not larger than
$\frac{8}{3}$. In addition, we show that the only ones having Chern ratio arbitrarily close to $\frac{8}{3}$ are the random surfaces associated
to the dual Hesse arrangement.

In fact, random surfaces provide a new record for Chern ratios of simply connected surfaces of general type. In $1996$, Persson, Peters and Xiao
\cite{PerPetXiao96} proved that the set of Chern ratios of simply connected surfaces of positive signature is dense in $[2,2.\overline{703}]$.
It is unknown the existence of such surfaces with Chern ratio in the range $(2.\overline{703},3)$, and so the results of Persson, Peters and
Xiao are the best known results in this direction. We remark that any surface $X$ of general type satisfies the Miyaoka-Yau inequality
$c_1^2(X)\leq 3 c_2(X)$, and equality holds if and only if $X$ is a ball quotient, and so its fundamental group is not trivial. In Section
\ref{s10}, we show that random surfaces give examples of simply connected surfaces in the unknown zone. The precise statement is the following.

\vspace{0.2 cm} \textbf{Theorem \ref{t101}.} There exist nonsingular simply connected projective surfaces of general type with Chern ratio
arbitrarily close to $\frac{71}{26}\approx 2.730769$. \vspace{0.3cm}

As one may expect, the scenario is very different when considering surfaces over fields $\K$ of positive characteristic. For instance, any
random surface $X$ associated to a line arrangement in $\P^2(\K)$ satisfies $c_1^2(X) \leq 3 c_2(X)$ (see Theorem \ref{t81}), and, in all
characteristics, we have examples for which their Chern ratio is arbitrarily close to $3$ (see Example \ref{e82}).

\vspace{0.1cm} Throughout this article: $p$ denotes a prime number. If $q$ is an integer with $0<q<p$, we denote by $q'$ the unique integer
satisfying $0<q'<p$ and $qq'\equiv 1 (\text{mod} \ p)$. For positive integers $\{a_1,\ldots,a_r\}$, we denote by $(a_1,\ldots,a_r)$ their
greatest common divisor. If $(a_1,\ldots,a_r)=1$, we call them coprime. For a real number $b$, let $[b]$ be the integral part of $b$, i.e., $[b]
\in \Z$ and $[b]\leq b < [b]+1$. We use the notation $\L^n := \L^{\otimes n}$ for a line bundle $\L$. We write $D \sim D'$ when the divisors $D$
and $D'$ are linearly equivalent. The normalization of a variety $W$ is denoted by $\overline{W}$. Projective curves and surfaces are assumed to
be irreducible.

\vspace{0.1cm} \textbf{Acknowledgments}: I am grateful to my Ph.D. thesis adviser I. Dolgachev, R.-P. Holzapfel, J. Kiwi, K. Girstmair, X.
Roulleau, and J. Tevelev for very valuable discussions. Part of this work was done at the Pontificia Universidad Cat\'olica de Chile. I wish to
thank the Department of Mathematics, and in particular Rub\'{\i} Rodriguez, for the hospitality.

\tableofcontents

\section{$p$-th root covers.} \label{s2}

The $p$-th root cover tool, which we are going to describe in this section, was introduced by H. Esnault and E. Viehweg in \cite{ViVanishing}
and \cite{Es82}. We will follow their approach in \cite[Ch. 3]{EsVi92}, where this tool is presented for varieties over fields of arbitrary
characteristic. Let $\K$ be an algebraically closed field, and let $p$ be a prime number with $p \neq$ Char$(\K)$.

Let $Y$ be a nonsingular projective surface over $\K$. Let $D$ be a nonzero effective divisor on $Y$ such that $D_{\text{red}}$ has simple
normal crossings, as defined in \cite[p. 240]{Laz1}. Let $D=\sum_{i=1}^r \nu_i D_i$ be its decomposition into prime divisors. By definition,
each $D_i$ is a nonsingular projective curve, and $D_{\text{red}}$ has only nodes as singularities.

Assume that there exist a line bundle $\L$ on $Y$ satisfying $$ \L^{p} \simeq \O_Y(D) .$$

We construct from the data $(Y,D,p,\L)$ a new nonsingular projective surface $X$ which represents a ``$p$-th root of $D$". Let $s$ be a section
of $\O_Y(D)$, having $D$ as zero set. This section defines a structure of $\O_Y$-algebra on $\bigoplus_{i=0} ^{p-1} \L^{-i}$ by means of the
induced injection $\L^{-p} \simeq \O_Y(-D) \hookrightarrow \O_Y$. The first step in this construction is given by the affine map $f_1: W
\rightarrow Y$, where $W:= \spec_Y \Big( \bigoplus_{i=0} ^{p-1} \L^{-i} \Big)$ (as defined in \cite[p. 128]{Hartshorne}).

Because of the multiplicities $\nu_i$'s, the surface $W$ might not be normal. The second step is to consider the normalization $\overline{W}$ of
$W$. Let $f_2: \overline{W} \rightarrow Y$ be the composition of $f_1$ with the normalization map of $W$. The surface $\overline{W}$ can be
explicitly described through the following key line bundles.

\begin{defi}
As in \cite{ViVanishing}, we define the line bundles $\L^{(i)}$ on $Y$ as $$ \L^{(i)}:= \L^i \otimes \O_Y\Bigl( - \sum_{j=1}^r \Bigl[\frac{\nu_j
\ i}{p}\Bigr] D_j \Bigr)$$ for $i \in{\{0,1,...,p-1 \}}$. \label{d21}
\end{defi}

\begin{prop} (see \cite[Cor. 3.11]{EsVi92})
The group $G= \Z/ p \Z$ acts on $\overline{W}$ (so that $\overline{W}/G = Y$), and on ${f_2}_* \O_{\overline{W}}$. Moreover, we have $$ {f_2}_*
\O_{\overline{W}} = \bigoplus_{i=0}^{p-1} {\L^{(i)}}^{-1}.$$ This is the decomposition of ${f_2}_* \O_{\overline{Y}}$ into eigenspaces with
respect to this action. \label{p21}
\end{prop}

Therefore, the normalization of $W$ is $$\overline{W}=\spec_Y \Big( \bigoplus_{i=0}^{p-1} {\L^{(i)}}^{-1} \Big).$$ Let us notice that the
multiplicities $\nu_i$'s can always be considered in the range $0 \leq \nu_i < p$. If we change multiplicities from $\nu_i$ to $\bar{\nu}_i$
such that $\bar{\nu}_i \equiv \nu_i (\text{mod} \ p)$ and $0 \leq \bar{\nu}_i < p$ for all $i$, then the corresponding varieties $\overline{W}$
will be isomorphic over $Y$ (see \cite{EsVi92} for example). Therefore, from now on, \underline{we assume $0< \nu_i < p$ for all $i$}.

The surface $\overline{W}$ may be singular, but its singularities are rather mild. They are toric surface singularities \cite[Ch. 5]{Mircea},
also called Hirzebruch-Jung singularities when the ground field is $\C$ \cite[p. 99-105]{BHPV04}. These singularities exactly occur over the
nodes of $D_{\text{red}}$. Let us assume that $D_i \cap D_j \neq \emptyset$ for some $i\neq j$, and consider a point $P \in D_i \cap D_j$. Then,
the construction above shows that the singularity at $f_2^{-1}(P) \in \overline{W}$ is isomorphic to the singularity of the normalization of
$$\spec \ \K[x,y,z]/(z^p-x^{\nu_i}y^{\nu_j}),$$ where $x$ and $y$ can be seen as local parameters on $Y$ defining $D_i$ and $D_j$ respectively.

We denote this isolated singularity by $T(p,\nu_i,\nu_j)$. One can easily check that it is isomorphic to the affine toric surface defined by the
vectors $(0,1)$ and $(p,-q)$ in $\Z^2$, where $q$ is the unique integer satisfying $\nu_i q+\nu_j \equiv 0 (\text{mod} \ p)$ and $0<q<p$
\cite[Ch. 5 p. 5-8]{Mircea}.

\begin{defi}
Let $0<a,b<p$ be integers, and let $q$ be the unique integer satisfying $a q+ b \equiv 0 (\text{mod} \ p)$ and $0<q<p$. Consider the
negative-regular continued fraction $$ \frac{p}{q} = e_1 - \frac{1}{e_2 - \frac{1}{\ddots - \frac{1}{e_s}}},$$ which we abbreviate as
$\frac{p}{q}=[e_1,...,e_s]$. For each isolated singularity $T(p,a,b)$, we define its \underline{length} as $l(q,p):=s$. This quantity is
symmetric with respect to $a,b$ (see Appendix). \label{d22}
\end{defi}

It is well-known how to resolve $T(p,\nu_i,\nu_j)$ by means of toric methods, obtaining the same situation as in the complex case (see \cite[Ch.
5 p. 5-8]{Mircea}). That is, if $\frac{p}{q}=[e_1,...,e_s]$ is the corresponding continued fraction in Definition \ref{d22}, then the
singularity $T(p,\nu_i,\nu_j)$ is resolved by a chain of $l(q,p)$ nonsingular rational curves $\{ E_1,\ldots, E_{l(q,p)} \}$ (see Figure
\ref{f1}), whose self-intersections are $E_i^2=-e_i$.

\begin{figure}[htbp]
\includegraphics[width=13cm]{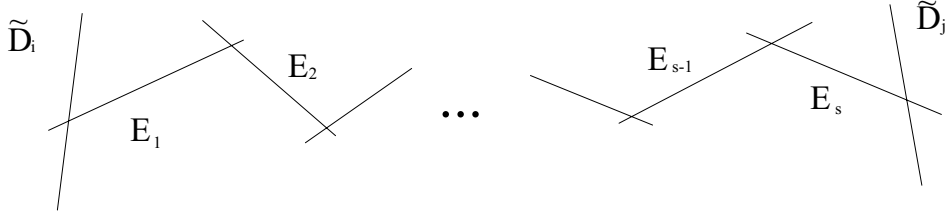}
\caption{Resolution over a point in $D_i \cap D_j$.} \label{f1}
\end{figure}

In this way, the surface $\overline{W}$ has only rational singularities. To see this, let $\ZZ$ be the fundamental cycle of the singularity
$T(p,\nu_i,\nu_j)$ (as defined in \cite{ArtinRatSing}). Hence, by definition, we have $\ZZ=\sum_{i=1}^s E_i$, where $E_i$'s are the
corresponding exceptional curves. In \cite{ArtinRatSing}, it is proved that a normal singularity is rational if and only if $p_a(\ZZ)=0$
(arithmetic genus of $\ZZ$ is zero). But $p_a(\ZZ)=p_a(\overline{\ZZ}) + s-1$, and $p_a(\overline{\ZZ})=1-s$, so the singularity is rational.

The third and last step is the minimal resolution $f_3: X \rightarrow \overline{W}$ of $\overline{W}$. The composition $f_2 \circ f_3$ is
denoted by $f: X \rightarrow Y$. The data $(Y,p,D,\L)$ \underline{uniquely} determines $X$.

\begin{prop}
The variety $X$ is a nonsingular irreducible projective surface, and
\begin{itemize}
\item[1.] There exist isomorphisms $\coh^j(X,\O_X) \simeq \bigoplus_{i=0}^{p-1} \coh^j \big(Y, {\L^{(i)}}^{-1} \big)$ for all $j$. \item[2.] If
$K_X$ and $K_Y$ are the canonical divisors for $X$ and $Y$ respectively, then we have the $\Q$-numerical equivalence $$ K_X \equiv f^*\Bigl( K_Y
+ \frac{p-1}{p} \sum_{i=1}^r D_i \Bigr) + \Delta,$$ where $\Delta$ is a $\Q$-divisor supported on the exceptional locus of $f_3$.
\end{itemize}
\label{p22}
\end{prop}

\begin{proof}
Since $\overline{W}$ has rational singularities, we have $R^b{f_3}_* \O_X=0$ for all $b>0$, and so $\coh^j(X,\O_X) \simeq \coh^j(\overline{W},
\O_{\overline{W}})$ for all $j$. But $f_2$ is affine and, by Proposition \ref{p21}, ${f_2}_* \O_{\overline{W}} = \bigoplus_{i=0}^{p-1}
{\L^{(i)}}^{-1}$. Therefore, $\coh^j(X,\O_X) \simeq \bigoplus_{i=0}^{p-1} \coh^j \big(Y, {\L^{(i)}}^{-1} \big)$.

If $X$ is connected, then it is irreducible because it is nonsingular. For $j=0$, we have $h^0(X,\O_X)= \sum_{i=0}^{p-1} h^0(Y,{\L^{(i)}}^{-1})=
1 + \sum_{i=1}^{p-1} h^0(Y,{\L^{(i)}}^{-1})$. Assume $X$ is not connected. Then, $h^0(X,\O_X)\geq 2$, and so there is $i$ such that
$h^0(Y,{\L^{(i)}}^{-1})>0$. In particular, ${\L^{(i)}}^{-1} \simeq \O_Y(H)$ where $H$ is an effective divisor. Hence, by intersecting $H$ with
curves $\Gamma_j$ such that $D_j.\Gamma_j>0$, we have that $\big[ \frac{\nu_j i}{p} \big] - \frac{\nu_j i}{p} = 0$ for all $j$, and so $i \nu_j
\equiv 0(mod \ p)$ for all $j$. This happens if and only if $(\nu_1,...,\nu_r,p)\neq 1$. But this is impossible, since $0< \nu_j <p$ for all
$j$.

Part $2.$ comes from the fact that $p\neq$ Char$(\K)$, and some local computations.
\end{proof}

\begin{rmk}
By using numerical properties of negative-regular continued fractions, one can prove that the $\Q$-divisor $f^*\Bigl( \frac{(p-1)}{p}
\sum_{i=1}^r D_i \Bigr) + \Delta$ is an effective $\Z$-divisor supported on $f^{-1}(D)$. This statement can be used to find $(-1)$- and
$(-2)$-curves in $X$ \footnote{A $(-1)$-curve ($(-2)$-curve) is a nonsingular rational curve with self-intersection equal to $-1$ ($-2$).}. In
several interesting examples, the surface $X$ will not be minimal. \label{r21}
\end{rmk}

\begin{rmk}
Consider the analogue situation in dimension one. Let $Y$ be a nonsingular projective curve over $\K$, and let $D=\sum_{j=1}^r \nu_i D_i$ be a
positive sum of points in $Y$ with $0< \nu_i < p$. Assume there is a line bundle $\L$ on $Y$ such that $\L^p \simeq \O_Y(D)$. Then, we can
perform the $p$-th root cover with data $(Y,p,D,\L)$, obtaining a nonsingular projective curve $X$. The multiplicities $\nu_i$'s do play a role
in the determination of the isomorphism class of $X$ (in \cite{Urzua1}, we worked out these isomorphism classes for certain curves over $\C$),
but they do not play any role in the determination of its genus (Riemann-Hurwitz formula \cite[IV.2]{Hartshorne}). This is not the case for
surfaces, mainly because the $D_i$ may intersect among each other. We will see that their numerical invariants are indeed affected by these
multiplicities. \label{r22}
\end{rmk}

\section{Log surfaces and their log Chern invariants.} \label{s3}

Log surfaces will encode the data of an arrangement of curves. We follow the point of view of Iitaka \cite[Ch. 11]{IitakaAG82}. Much more
information about log surfaces can be found in Iitaka \cite{IiLines78, IitakaBirGeo81}, Sakai \cite{Sakai1}, Kobayashi \cite{Kobayashi1}, and
Miyanishi \cite{Miya01}.

\begin{defi}
Let $Y$ be a nonsingular projective surface over $\K$, and let $\nA$ be an effective \underline{simple normal crossing} divisor on $Y$ (as
defined in \cite[p. 240]{Laz1}, abbreviated SNC). The pair $(Y,\nA)$ is called \underline{log surface}. \label{d31}
\end{defi}

A log surface $(Y,\nA)$ can be thought as the open surface $Y \setminus \nA$ which wants to remember its compactification by $\nA$. In this way,
one studies $Y\setminus \nA$ through the projective surface $Y$ with a modified sheaf of differentials.

\begin{defi} (see  \cite[p. 321]{IitakaAG82})
The sheaf of \underline{log differentials} along $\nA$, denoted by $\Omega_Y^1(\log \nA)$, is the $\O_Y$-submodule of $\Omega_Y^1 \otimes
\O_Y(\nA)$ satisfying
\begin{itemize}
\item[(i)] $\Omega_Y^1(\log \nA)|_{Y\setminus \nA}=\Omega_{Y\setminus \nA}^1$. \item[(ii)] At any point $P$ in $\nA$, we have $\omega_p \in
\Omega_Y^1(\log \nA)_P$ if and only if $\omega_p = \sum_{i=1}^m a_i \frac{dy_i}{y_i} + \sum_{j=m+1}^{2} b_j dy_j$, where $(y_1,y_2)$ is a local
system around $P$ for $Y$, and $\{ y_1\cdots y_m=0 \}$ defines $\nA$ around $P$.
\end{itemize}
\label{d32}
\end{defi}

Hence, $\Omega_Y^1(\log \nA)$ is locally free sheaf of rank two. As in the projective case, we define the \underline{log canonical divisor} as
$K_Y + \nA$, since $\bigwedge^2 \Omega_Y^1(\log \nA) \simeq \O_Y(K_Y + \nA)$. Various log invariants are defined in analogy to the projective
case (see for example \cite{IiLines78}). We are interested in the ones coming from Chern invariants.

\begin{defi}
The log Chern classes of the log variety $(Y,\nA)$ are defined as $\bar{c}_i(Y,\nA):= c_i({\Omega_Y^1(\log \nA)}^*)$ for $i= 1,2$; and the
corresponding \underline{log Chern numbers} as \begin{center} $ \bar{c}_1^2(Y,\nA):= c_1 \big({\Omega_Y^1(\log \nA)}^{*} \big)^2 \ \ \ \
\text{and} \ \ \ \ \bar{c}_2(Y,\nA):= c_2 \big({\Omega_Y^1(\log \nA)}^{*} \big).$ \end{center} \label{d33}
\end{defi}

We now show explicit combinatorial formulas for these numbers.

\begin{prop}
Consider a nonsingular log surface $(Y,\nA)$. Let $\nA= \sum_{i=1}^r D_i$ be the decomposition of $\nA$ into prime divisors. Let $g(D_i)=
\text{dim}_{\K} \coh^1(D_i,\O_{D_i})$, and let $t_2$ be the number of nodes of $\nA$. Then, the log Chern numbers of $(Y,\nA)$ are
$$\bar{c}_1^2(Y,\nA)= c_1^2(Y) - \sum_{i=1}^r D_i^2 + 2 t_2 + 4 \sum_{i=1}^r (g(D_i)-1),$$ and $\bar{c}_2(Y,\nA)= c_2(Y) + t_2 + 2 \sum_{i=1}^r
(g(D_i)-1)$. \label{p31}
\end{prop}

\begin{proof}
By definition, $\bar{c}_1^2(Y,\nA)= K_Y^2 + 2 \sum_{i=1}^r K_Y.D_i + \sum_{i=1}^r D_i^2 + 2 \sum_{i<j} D_i.D_j$. Then, by the adjunction formula
\cite[p. 361]{Hartshorne}, the result follows.

We compute the second log Chern number via the Hirzebruch-Riemann-Roch Theorem \cite[p. 432]{Hartshorne}. We apply this Theorem to obtain the
equality $$\chi(Y,\Omega_Y^1(\log \nA)^*)= \deg \big(\text{ch}(\Omega_Y^1(\log \nA)^*).\text{td}({\Omega_Y^1}^*) \big).$$ Then, we consider the
residual exact sequence $$0 \rightarrow \Omega_Y^1 \rightarrow \Omega_Y^1(\log \nA) \rightarrow \oplus_{i=1}^{r} \O_{D_i} \rightarrow 0,$$ to
compute $\chi(Y,\Omega_Y^1(\log \nA)^*)= \sum_{i=1}^r \chi(D_i, \Omega_Y^2 \otimes \O_{D_i}) + \chi(Y,{\Omega_Y^1}^*)$. Finally, by using the
Riemann-Roch Theorem for curves and by developing the summands involved, we find the formula for $\bar{c}_2(Y,\nA)$.
\end{proof}

\section{Numerical invariants of $p$-th root covers.} \label{s4}

In this section, we compute the Chern numbers of surfaces $X$ coming from $p$-th root covers with data $(Y,D,p,\L)$ (as in Section \ref{s2}). As
before, let $p$ be a prime number. Let $Y$ be a nonsingular projective surface over $\K$, and let $D$ be a nonzero effective divisor on $Y$. We
write its prime decomposition as $D=\sum_{i=1}^r \nu_i D_i$. Assume that $D_{\text{red}}$ is a simple normal crossing divisor, and that
$0<\nu_i<p$ for all $i$. Assume that there exists a line bundle $\L$ on $Y$ such that $\L^p \simeq \O_{Y}(D)$. Let $f:X \rightarrow Y$ be the
$p$-th root cover over $Y$ along $D$.

By Proposition \ref{p22}, the Riemann-Roch Theorem and Serre's duality, we have
$$q(X)=q(Y) + \sum_{i=1}^{p-1} h^1(Y, {\L^{(i)}}^{-1}), \ \ \ p_g(X) = p_g(Y) + \sum_{i=1}^{p-1} h^0(Y, \Omega_Y^2 \otimes \L^{(i)}),$$ and
$$\chi(X,\O_X)= p \chi(Y,\O_Y) + \frac{1}{2} \sum_{i=1}^{p-1} \L^{(i)}.(\L^{(i)}\otimes \Omega_Y^2).$$

Let us develop a little more the expression for $\chi(X,\O_X)$. We have
$$ \sum_{i=1}^{p-1} {\L^{(i)}}^{2} = \sum_{i=1}^{p-1} \frac{i^2}{p^2}
D^2 - \sum_{i=1}^{p-1} \frac{2i}{p} \Big(\sum_{j=1}^{r} \Bigl[\frac{\nu_j i}{p}\Bigr] D_j.D \Big) + \sum_{i=1}^{p-1} \Big(\sum_{j=1}^r
\Bigl[\frac{\nu_j i}{p}\Bigr] D_j \Big)^2,$$ and so
$$ \sum_{i=1}^{p-1} {\L^{(i)}}^{2} = \frac{(p-1)(2p-1)}{6p} \sum_{j=1}^{r} D_j^2 \ +$$ $$ \sum_{j< k} \Bigl(
\frac{(p-1)(2p-1)(\nu_k^2 + \nu_j^2)}{6p \nu_j \nu_k} - \sum_{i=1}^{p-1} \Bigl[\frac{\nu_j i}{p}\Bigr]^2 \frac{\nu_k}{\nu_j} - \sum_{i=1}^{p-1}
\Bigl[\frac{\nu_k i}{p}\Bigr]^2 \frac{\nu_j}{\nu_k} + 2 \sum_{i=1}^{p-1} \Bigl[\frac{\nu_j i}{p}\Bigr] \Bigl[\frac{\nu_k i}{p}\Bigr]
\Big)D_j.D_k,$$ and $$\sum_{i=1}^{p-1} \L^{(i)}.\Omega_Y^2 =  \sum_{j=1}^{r} \sum_{i=1}^{p-1} \Bigl(\frac{\nu_j i}{p} - \Bigl[\frac{ \nu_j
i}{p}\Bigr]\Bigr)D_j.K_Y  = \frac{p-1}{2} \sum_{j=1}^{r} D_j.K_Y.$$

\begin{defi}
Let $p$ be a prime number. Let $q$ be an integer such that $0<q<p$. The \underline{Dedekind sum} associated to the pair $(q,p)$ is defined as
$$s(q,p):= \sum_{i=1}^{p-1} \Bigl(\Bigl(\frac{i}{p}\Bigr)\Bigr)\Bigl(\Bigl(\frac{iq}{p}\Bigr)\Bigr)$$ where $((x))=x-[x]-\frac{1}{2}$ for any
rational number $x$. \label{d41}
\end{defi}

Connections between Dedekind sums and geometry can be found in \cite{HiZa74}. For us, these sums naturally appear when considering the
Riemann-Roch Theorem. The number $s(q,p)$ depends only on the class of $q$ modulo $p$. Also, $s(p-q,p)=-s(q,p)$ and $s(q',p)=s(q,p)$, where $q'$
is the unique integer satisfying $0< q'<p$ and $qq'\equiv 1 ($mod $p)$. We now prove that Dedekind sums precisely measure the effect of the
multiplicities $\nu_i$'s in $\chi(X,\O_X)$.

\begin{prop}
Let $f:X \rightarrow Y$ be the $p$-th root cover associated to the data $(Y,D,p,\L)$. Let $t_2$ be the number of nodes of $D_{\text{red}}$.
Then, $$ \chi(X,\O_X) = p \chi(Y,\O_Y) - \frac{p^2-1}{12p} \sum_{i=1}^r D_i^2 + \frac{p-1}{4} \big(t_2 + 2 \sum_{i=1}^r (g(D_i)-1) \big)  -
\sum_{i<j} s(p-\nu_i' \nu_j,p) D_i.D_j.$$ \label{p41}
\end{prop}

\begin{proof}
We temporarily define $S(a,b;p):= \sum_{i=1}^{p-1} \bigl[\frac{a i}{p}\bigr] \bigl[\frac{b i}{p}\bigr]$ for any integers $a,b$ satisfying
$0<a,b<p$. Then, since $\sum_{i=1}^{p-1} \Big(a i - \Bigl[ \frac{ a i}{p} \Bigr] p \Big)^2 = \sum_{i=1}^{p-1} i^2 = \frac{p(p-1)(2p-1)}{6}$, one
can check that $ \sum_{i=1}^{p-1} i \Bigl[\frac{ai}{p}\Bigr] = \frac{1}{12a} (a^2-1)(p-1)(2p-1) - \frac{p}{2a} S(a,a;p).$

One can easily verify (see for example \cite[p. 94]{HiZa74}) that $ s(a,p) = \frac{p-1}{6p}(2ap -a -\frac{3}{2}p) - \frac{1}{p} \sum_{i=1}^{p-1}
i \Bigl[\frac{ai}{p}\Bigr]$, and so $$ s(a,p) = \frac{1}{12ap} (p-1)(2pa^2-a^2-3ap+2p-1) - \frac{1}{2a} S(a,a;p).$$

On the other hand, we have
$$ S(a,b;p)= s(a'b,p) - as(b,p) - bs(a,p) + \frac{p-1}{12p} \big(3p-3pa-3pb+2ab(2p-1) \big).$$

Putting all together, $$-\frac{a}{b}S(b,b;p) - \frac{b}{a} S(a,a;p) + 2 S(a,b;p) = \frac{1-p}{6abp} \big(a^2(2p-1) + b^2(2p-1) -3abp \big) +
2s(a'b;p).$$

We now replace these expressions (taking $a=\nu_i$, $b=\nu_j$) in the formula for $\sum_{i=1}^{p-1} {\L^{(i)}}^{2}$ above. Finally, we use the
adjunction formula to include $g(D_i)$'s.
\end{proof}

\begin{defi}
Let $p$ be a prime number, and let $q$ be an integer with $0<q<p$. Consider the negative-regular continued fraction
$\frac{p}{q}=[e_1,\ldots,e_{l(q,p)}]$. We define the \underline{canonical part} of the pair $(q,p)$ as $$c(q,p):= \frac{q+q'}{p} +
\sum_{i=1}^{l(q,p)}(e_i-2).$$ \label{d42}
\end{defi}

To each point in $D_i \cap D_j$ with $i<j$, we associate the negative-regular continued fraction $\frac{p}{q}=[e_1,\ldots,e_{l(q,p)}]$, where
$q=p-\nu'_i \nu_j$ (the numbers $\nu'_i \nu_j$ are always taken in $\{1,2,\ldots,p-1\}$). Hence, each node of $D_{\text{red}}$ has its
corresponding canonical part $c(p-\nu'_i \nu_j,p)$.

\begin{prop}
Let $f:X \rightarrow Y$ be the $p$-th root cover associated to the data $(Y,D,p,\L)$. Let $t_2$ be the number of nodes of $D_{\text{red}}$.
Then, $$ c_1^2(X) = p \bar{c}_1^2(Y,D_{\text{red}}) - 2\Big(t_2 + 2 \sum_{i=1}^r (g(D_i)-1) \Big) + \frac{1}{p} \sum_{i=1}^r D_i^2 - \sum_{i<j}
c(p-\nu_i' \nu_j,p) D_i.D_j.$$ \label{p42}
\end{prop}

\begin{proof}
As we saw in Proposition \ref{p22}, we have the $\Q$-numerical equivalence $$K_X \equiv f^*\Bigl( K_Y + \frac{p-1}{p} \sum_{i=1}^r D_i \Bigr) +
\Delta,$$ where $\Delta$ is a $\Q$-divisor supported on the exceptional locus of $f_3$. Let $\{P_1,\ldots, P_{t_2} \}$ be the set of nodes of
$D_{\text{red}}$. We know that, in the $p$-th root cover procedure, the singularity coming from $P_k$ is resolved by a chain $\{E_{1,k}, \ldots,
E_{s,k} \}$ of $\P^1(\K)$'s. If $\frac{p}{q_k}=[e_1,\ldots,e_{l(q_k,p)}]$, where $P_k \in D_i \cap D_j$ and $q_k=p-\nu'_i \nu_j$ , then
$s=l(q_k,p)$ and $E_{i,k}^2=-e_{i,k}$ for all $i$. Therefore, we can write $\Delta = \sum_{k=1}^{t_2} \Delta_k$, where $\Delta_k=
\sum_{i=1}^{l(q_k,p)} \alpha_{i,k} E_{i,k}$. In this way, we compute
$$c_1^2(X)= p c_1^2(Y) - \frac{p^2-1}{p} \sum_{i=1}^r D_i^2 + 4(p-1) \sum_{i=1}^r (g(D_i)-1) + 2\frac{(p-1)^2}{p} t_2 + \sum_{k=1}^{t_2}
\Delta_k^2.$$ It is not hard to see that $\Delta_k^2 = \sum_{i=1}^{l(q_k,p)} \alpha_{i,k} (e_{i,k}-2)$. By Lemma \ref{ap2}, we have $$\Delta_k^2
= \sum_{i=1}^{l(q_k,p)} (2-e_{i,k}) - \frac{q_k+q_k'}{p} + 2 \frac{p-1}{p}.$$ We finally replace these sums above, and rearrange terms (and use
Proposition \ref{p31}).
\end{proof}

\begin{example}
Consider $0<q<p$, and the corresponding $\frac{p}{q}=[e_1,\ldots,e_{l(q,p)}]$. For now, we assume $\K=\C$. Let $Y=\P^2(\C)$ and $D= L_1+\ldots
+L_{r-1}+(p-q)L_r$ where $\{L_1,\ldots,L_r \}$ is a general line arrangement (in particular, only nodes as singularities), and $r=q+1$. Hence,
$\O_Y(1)^p \simeq \O_Y(D)$, and so we have the $p$-th root cover $f:X \rightarrow Y$ from the data $(Y,p,D,\L=\O_Y(1))$. Then, we apply
Propositions \ref{p41} (and $s(1,p)=\frac{(p-1)(p-2)}{12p}$) to compute
$$ \chi(X,\O_X)= p - \frac{(p^2-1)r}{12p} - \frac{1}{8}(p-1)r(5-r) + \frac{1}{24p} (r-1)(r-2)(p-1)(p-2) + (r-1) s(p-q,p).$$

We now use the underlying complex topology to find the topological Euler Characteristic $\chi_{\text{top}}(X)$ of $X$. The following is a
well-known topological lemma: if $B$ be a complex projective variety and $A\subseteq B$ a subvariety such that $B\setminus A$ is nonsingular,
then $\chi_{\text{top}}(B) = \chi_{\text{top}} (A) + \chi_{\text{top}}(B\setminus A)$. By repeatedly applying this lemma, and by using the fact
$c_2(X)=\chi_{\text{top}}(X)$, we find $$ c_2(X) = 3p + (1-p) \frac{r(5-r)}{2} + \frac{(r-1)(r-2)}{2}(p-1)+ (r-1) l(q,p).$$ Using Proposition
\ref{p42}, we find $c_1^2(X)$. Finally, the Noether's formula $12 \chi(X,\O_X) = c_1^2(X) + c_2(X)$ gives us the relation $$12s(q,p) -
\sum_{i=1}^{l(q,p)} e_i + 3l(q,p) = \frac{q+q'}{p}.$$

We notice that this formula was found by Holzapfel in \cite[Lemma 2.3]{Hol88} using the original definition of Dedekind sums via Dedekind
$\eta$-function. For a similar formula involving regular continued fractions see \cite{Ba77,Hicke77}. A direct consequence is the computation of
$c_2(X)$. \label{e41}
\end{example}

\begin{prop}
Let $f:X \rightarrow Y$ be the $p$-th root cover associated to the data $(Y,D,p,\L)$. Let $t_2$ be the number of nodes of $D_{\text{red}}$.
Then, $$c_2(X) = p \bar{c}_2(Y,D_{\text{red}}) - \Big(t_2 + 2 \sum_{i=1}^r (g(D_i)-1) \Big) + \sum_{i<j} l(p-\nu_i' \nu_j,p) D_i.D_j.$$
\label{p43}
\end{prop}

\begin{proof}
We compute $c_2(X)$ via Noether's formula, and Propositions \ref{p41} and \ref{p42}. To include the terms $l(p-\nu_i' \nu_j,p)$, we use the
formula in Example \ref{e41}, which reads $$12s(p-\nu'_i \nu_j,p) + l(p-\nu'_i \nu_j,p) = c(p-\nu'_i \nu_j,p).$$ Finally, the log Chern number
$\bar{c}_2(Y,D_{\text{red}})$ appears from Proposition \ref{p31}.
\end{proof}

\section{Simple crossing divisible arrangements.} \label{s5}

In this section, we define the key arrangements $\A$ which will produce surfaces $X$ via $p$-th root covers branch along $\nA$, the minimal log
resolution of $\A$. The divisor $D$ in Section \ref{s4} enters to the picture via the equality $D_{\text{red}}=\nA$.

Let $d\geq 3$ be an integer, and let $Z$ be a nonsingular projective surface over $\K$. An \underline{arrangement of $d$ curves} $\A$ in $Z$ is
a set $\{C_1,\ldots,C_d \}$ of $d$ nonsingular projective curves such that $\bigcap_{i=1}^d C_i = \emptyset$. We loosely consider $\A$ as the
set $\{ C_1,\ldots,C_d \}$, or as the divisor $C_1 + \ldots + C_d$, or as the curve $\bigcup_{i=1}^d C_i$. We say that $\A$ is defined over
$\mathbb{L} \subseteq \K$ if all the curves in $\A$ are defined over $\mathbb{L}$.

\begin{defi}
An arrangement of $d$ curves $\A$ in $Z$ is \underline{simple crossing} if $C_i$ and $C_j$ intersect transversally for all $i\neq j$. For
$1<n<d$, an \underline{$n$-point} of $\A$ is a point in $\A$ contained in exactly $n$ curves of $\A$. The \underline{number of $n$-points} of
$\A$ is denoted by $t_n$. \label{d51}
\end{defi}

\begin{defi}
An arrangement of $d$ curves $\A$ in $Z$ is \underline{divisible} if $\A$ splits into $v\geq 1$ arrangements of $d_i$ curves $\A_i$ (so $d_i
\geq 3$) satisfying:
\begin{itemize}
\item[1.] For all $i\neq j$, $\A_i$ and $\A_j$ are disjoint as sets of curves. \item[2.] For each $i \in \{1,\ldots, v\}$, there exists a line
bundle $\L_i$ on $Z$ such that, for each $C$ in $\A_i$, we have $\O_Z(C) \simeq \L_i^{u(C)}$ for some integer $u(C)>0$.
\end{itemize}
Given $i \in \{1,\ldots, v\}$, we can and do assume that the corresponding $u(C)$'s are coprime. \label{d52}
\end{defi}

\begin{example}
Consider $Z=\P^2(\K)$, and arrangements of $d$ lines $\A=\{L_1,\ldots,L_d \}$ in $Z$. We recall that, by definition, $d\geq 3$ and $t_d=0$.
These arrangements are simple crossing, and satisfy ${d \choose 2} = \sum_{n \geq 2} {n \choose 2} t_n $. In addition, they are divisible by
taking $v=1$, $\L_1=\O_Z(1)$, and $u(L_i)=1$ for all $i$. \label{e51}
\end{example}

\begin{example}
Consider $Z=\P^2(\K)$, and an arrangement of $d$ nonsingular plane curves $\A=\{C_1,\ldots,C_d \}$. Let $m= \big(\deg(C_1),\ldots,\deg(C_d)
\big)$. Then, $\A$ is divisible by taking $v=1$, $\L_1=\O_{Z}(m)$, and $u(C_i)=\frac{\deg(C_i)}{m}$. Of course, they may or may not be simple
crossing. \label{e52}
\end{example}

\begin{example}
For examples with $v>1$, we take $Z= \P^1(\K) \times \P^1(\K)$. Let us denote the classes of the Picard group of $Z$ by $\O_Z(a,b)$. Assume we
have three simple crossing arrangements, defined as $\A_1=\{A_1,\ldots,A_{d_1} \}$ with $\O_Z(A_i) \simeq \O_Z(1,0)$, $\A_2=\{B_1,\ldots,B_{d_2}
\}$ with $\O_Z(B_i) \simeq \O_Z(0,1)$, and $\A_3=\{C_1,\ldots,C_{d_3} \}$ with $\O_Z(C_i) \simeq \O_Z(1,1)$. Then, the arrangement $\A=\A_1 \cup
\A_2 \cup \A_3$ is naturally a simple crossing divisible arrangement of $d$ curves, with $v=3$, $d=d_1+d_2+d_3$, $\L_1=\O_Z(1,0)$,
$\L_2=\O_Z(0,1)$, $\L_3=\O_Z(1,1)$, and $u(C)=1$ for all curve $C$ in $\A$. \label{e53}
\end{example}

Let $\A$ be a simple crossing arrangement in $Z$. Each pair $(Z,\A)$ produces a unique nonsingular log surface $(Y,\nA)$ by performing blow-ups
at all the $n$-points of $\A$ with $n\geq 3$. In this way, if $\sigma: Y \rightarrow Z$ is the corresponding blow-up map, we define $\nA:=
\sigma^{*}(\A)_{\text{red}}$. The SNC divisor $\nA$ contains the proper transforms of the curves in $\A$, and also the exceptional divisors over
each $n$-point with $n \geq 3$. We said that $(Y,\nA)$ is the \underline{associated pair} of $(Z,\A)$, and we write down the prime decomposition
of $\nA$ as $\sum_{i=1}^r D_i$. In this way, we define the \underline{log Chern numbers} associated to $(Z,\A)$ as the log Chern numbers of the
associated pair $(Y,\nA)$. We denote them as $\bar{c}_1^2(Z,\A):= \bar{c}_1^2(Y,\nA)$ and $\bar{c}_2(Z,\A):= \bar{c}_2(Y,\nA)$.

\begin{prop}
Let $\A= \{C_1,\ldots, C_d \}$ be a simple crossing arrangement in $Z$. Then, $$\bar{c}_1^2(Z,\A) = c_1^2(Z) - \sum_{i=1}^d C_i^2 + \sum_{n\geq
2} (3n-4)t_n + 4 \sum_{i=1}^d (g(C_i)-1),$$ and $\bar{c}_2(Z,\A) = c_2(Z) + \sum_{n\geq 2} (n-1)t_n + 2 \sum_{i=1}^d (g(C_i)-1)$. \label{p51}
\end{prop}

\begin{proof}
This is a straightforward application of Proposition \ref{p31}.
\end{proof}

\section{Divisible arrangements and partitions of prime numbers.} \label{s6}

Let $Z$ be a nonsingular projective surface over $\K$, and let $\A= \{C_1,\ldots, C_d \}$ be a simple crossing divisible arrangement in $Z$. In
this section, we produce surfaces $X$ from pairs $(Z,\A)$ via $p$-th root covers. Let $(Y,\nA)$ be the associated pair of $(Z,\A)$, and let
$\sigma: Y \rightarrow Z$ be the corresponding blow-up map (as in the previous section). Since $\A$ is a divisible arrangement, it decomposes as
a disjoint union (as sets) of subarrangements $\A_1 \sqcup \ldots \sqcup A_v$. We write $\A_i = \{ C_{1,i},\ldots, C_{d_i,i} \}$. Also by the
definition of divisible arrangement, there is a line bundle $\L_i$ and positive integers $u(C_{j,i})$ such that $$\O_Z(C_{j,i}) \simeq
\L_i^{u(C_{j,i})}.$$ We recall that $\big( u(C_{1,i}),\ldots,u(C_{d_i,i}) \big) =1$ for all $i \in \{1,\ldots,v\}$.

Let $p$ be a prime number. Consider the Diophantine linear system of equations $\S(\A)$,

\vspace{0.3cm}
\begin{center} \begin{tabular}{c}
$u(C_{1,1})\mu_{1,1} + u(C_{2,1})\mu_{2,1} + \ldots + u(C_{d_1,1})\mu_{d_1,1} =p$ \\
$u(C_{1,2})\mu_{1,2} + u(C_{2,2})\mu_{2,2} + \ldots + u(C_{d_2,2})\mu_{d_2,2} =p$ \\
$\vdots$ \\
$u(C_{1,v})\mu_{1,v} + u(C_{2,v})\mu_{2,v} + \ldots + u(C_{d_v,v})\mu_{d_v,v} =p.$
\end{tabular}
\end{center}
\vspace{0.3cm}

We think about this system as a bunch of weighted partitions of $p$. For $p$ large enough, $\S(\A)$ has solutions. Actually, it is well-known
that the number of positive integer solutions is equal to (see \cite{BGK01}, for example) $$ \prod_{i=1}^v \Big( \frac{p^{d_i-1}}{(d_i-1)!
u(C_{1,i})u(C_{2,i})\cdots u(C_{d_i,i})} + O(p^{d_i-2}) \Big).$$

Consider a positive solution $\{ \mu_{i,j} \}$ of $\S(\A)$. Notice that $0< \mu_{i,j} < p$ for all $i,j$. On $Y$, we define the divisor $$D:=
\sigma^{*} \Big( \sum_{i=1}^v \sum_{j=1}^{d_i} \mu_{i,j} C_{i,j} \Big).$$ Let us also define the line bundle $\L$ on $Y$ as $$\L := \sigma^*
\Big( \L_1 \otimes \L_2 \otimes \cdots \otimes \L_v \Big). $$ Then, because of $\S(\A)$, we have $$\L^p \simeq \O_Y(D).$$ Let $D= \sum_{i=1}^r
\nu_i D_i$ be the prime decomposition of $D$, and assume $0 < \nu_i < p$. Since we want to have all the exceptional divisors of $\sigma$ in $D$,
we consider solutions $\{ \mu_{i,j} \}$ of $\S(\A)$ which have $0 < \nu_i < p$ for all $i$. This is always possible for large primes.

If $D_i$ is the proper transform of $C_i$, then $\nu_i = \mu_i$. When $D_i$ is the exceptional divisor over an $n$-point, say in $C_{j_1} \cap
\cdots \cap C_{j_n}$, then $\nu_i = \mu_{j_1} + \ldots + \mu_{j_n}$ modulo $p$.

In this way, we perform the $p$-th root cover $f: X \rightarrow Y$ with data $(Y,p,D,\L)$ coming from $(Z,\A)$. In the next section, we will
prove that random solutions (partitions) of $\S(\A)$ (of primes $p$) produce surfaces with interesting properties.

\section{Projective surfaces vs. log surfaces via random partitions.} \label{s7}

\begin{thm} Let $Z$ be a nonsingular projective surface over $\K$, and let $\A$ be a simple crossing divisible
arrangement of curves in $Z$. Assume $\bar{c}_2(Z,\A) \neq 0$. Then, there exist nonsingular projective surfaces $X$ with

\vspace{0.2 cm}
\begin{center} {\Large$\frac{c_1^2(X) }{c_2(X)}$} arbitrarily close to {\Large $\frac{\bar{c}_1^2(Z,\A)}{\bar{c}_2(Z,\A)}$}. \end{center}

\label{t71}
\end{thm}

\begin{proof} Consider the construction in Section \ref{s6}. Let us re-index the multiplicities $\mu_{i,j}$ as
$\mu_{j + \sum_{k=1}^{i-1} d_{k}}:= \mu_{i,j}$ (with $d_0:=0$), to simplify notation . Our construction is summarized in the following diagram
$$ X \stackrel{\rm \text{f}} {\rightarrow} Y \stackrel{\rm \sigma} {\rightarrow} Z .$$

The Chern numbers of $X$ are computed in Propositions \ref{p42} and \ref{p43}. The formulas of $c_1^2(X)$ and $c_2(X)$ contain the ``error
terms" \begin{center} $CCF := \sum_{i<j} c(p-\nu_i' \nu_j,p) D_i.D_j \ \ $ and $ \ \ LCF:= \sum_{i<j} l(p-\nu_i' \nu_j,p) D_i.D_j$.
\end{center} We are going to prove the existence of ``good" weighted partitions $\{ \mu_{i} \}$ for arbitrarily large primes $p$, which make
$\frac{CCF}{p}$ and $\frac{LCF}{p}$ arbitrarily small. In addition, we will show that random partitions are ``good", with probability
approaching $1$ as $p$ becomes arbitrarily large. The key numbers to study are the $p - \nu_i' \nu_j$, which are defined for every node of
$D_{\text{red}}$. In terms of $\mu_i$'s, these numbers are equal to either $p-\mu'_i \mu_j$ or $p-\mu'_i (\mu_{j_1}+ \cdots + \mu_{j_n})$.

Let $\FF \subset \{0, \ldots , p-1 \}$ be the \underline{bad set} (Definition \ref{ap3}). Let $b(\mu'_i \mu_j)$ be the set of solutions of
$\S(\A)$ having $p-\mu'_i \mu_j$ in $\FF$ for fixed $i\neq j$; and similarly $b \big(\mu'_i (\mu_{j_1}+ \cdots + \mu_{j_n}) \big)$ be the set of
solutions of $\S(\A)$ having $p- \mu'_i \big(\mu_{j_1}+ \ldots +\mu_{j_n} \big)$ in $\FF$ for fixed $i,j_1,\ldots,j_n$, having $2<n<d$ and
$i=j_k$ for some $k$. We define the set of \underline{bad solutions} $\BB$ of $\S(\A)$ as the union of $b(\mu'_i \mu_j)$'s and $b \big(\mu'_i
(\mu_{j_1}+ \cdots + \mu_{j_n}) \big)$'s, over all allowed indices. We want to bound the size of $\BB$.

Let $p$ be a large prime number. We first consider $v=1$, that is, $\S(\A)$ consists of one equation. Let us write it down as \begin{center}
$\S(\A) : \ \ \ u_1 \mu_1+ \ldots + u_d \mu_d =p$. \end{center} (recall that $(u_1,\ldots,u_d)=1$). As we said before, the number of positive
integer solutions of $\S(\A)$ is $\frac{p^{d-1}}{(d-1)!u_1 u_2 \cdots u_d} + O(p^{d-2})$. In general, we denote the \underline{number of
positive integer solutions} of $b_1 x_1 + \ldots + b_m x_m = a$ by $\alpha_m^{b_1,\ldots,b_m}(a)$.

In order to bound bad solutions, we consider the following two cases.

\vspace{0.1cm}

\textbf{(1)} For simplicity, assume $i=1$ and $j=2$. For each fixed pair $(\mu_1,\mu_2) \in {\Z/p\Z}^* \times {\Z/p\Z}^*$, we have
$\alpha_{d-2}^{u_3,\ldots,u_d}(p-u_1 \mu_1 - u_2 \mu_2)$ solutions of $\S(\A)$. For a fixed $0<\mu_1<p$, we consider the bijection $\varphi:
\Z/p\Z \rightarrow \Z/p\Z$, defined as $\varphi(\mu_2)=-\mu'_1 \mu_2$. Some pairs $(\mu_1, -\mu'_1 \mu_2)$ have $-\mu'_1 \mu_2 \in \FF$, giving
$\alpha_{d-2}^{u_3,\ldots,u_d}(p- u_1 \mu_1 - u_2 \mu_2)$ bad solutions to $\S(\A)$. We know that for every pair $(\mu_1,\mu_2)$, there exists a
positive number $M_{i,j}$ (independent of $p$) such that $\alpha_{d-2}^{u_3,\ldots,u_d}(p- u_1 \mu_1 - u_2 \mu_2) < M_{i,j} p^{d-3}$. Therefore,
we have $$|b(\mu'_i \mu_j)|<p \cdot |\FF| M_{i,j} p^{d-3}.$$

\vspace{0.1cm}

\textbf{(2)} Assume $i=1$ and $j_1=1, \ldots, j_n=n$ (we have $2<n<d$ since there are no $d$-points in $\A$ by definition). For a fixed $0 < \mu
< p- u_1 \mu_1$, by definition, the number of solutions of $u_2 \mu_2 + \ldots + u_n \mu_n = \mu$ is $\alpha_{n-1}^{u_2,\ldots,u_n}(\mu)$. In
this way, there is $M_{j_1,\ldots,j_n}$ (independent of $p$) such that $\alpha_{n-1}^{u_2,\ldots,u_n}(\mu) < M_{j_1,\ldots,j_n} p^{n-2}$. Also,
when we fixed $0< \mu_1, \mu_{j_k} < p$, we have $\alpha_{d-n}^{u_{n+1},\ldots,u_d}(p-u_1 \mu_1 - \ldots - u_n \mu_n) < N_{j_1,\ldots,j_n}
p^{d-n-1}$ associated solutions of $\S(\A)$ for some constant $N_{j_1,\ldots,j_n}$ (independent of $p$). By applying the bijection of
\textbf{(1)} for $(\mu_1,\mu)$, we conclude $$\big| b \big(\mu'_i (\mu_{j_1}+ \ldots +\mu_{j_n}) \big) \big| < p \cdot |\FF|  M_{j_1,\ldots,j_n}
p^{n-2} N_{j_1,\ldots,j_n} p^{d-n-1}.$$

\vspace{0.1cm}

Therefore, the number of bad solutions satisfies $|\BB|< |\FF| M_0 p^{d-2}$, where $M_0$ is a positive number which depends on $u_i$'s and $d$,
and all possible combinations of pairs $i,j$ and tuples $i_1,\ldots,i_k,j$ as above, but it does not depend on $p$. Since $p$ is a large prime,
the exact number of solutions of $\S(\A)$ is $\frac{p^{d-1}}{(d-1)!u_1 u_2 \cdots u_d} + O(p^{d-2})$. On the other hand, by Theorem \ref{ap5},
we know that $|\FF|< \sqrt{p} \big(\log(p)+2\log(2) \big)$ (in there, we take $C=1$). In this way, $$\frac{|\BB|}{\alpha_d^{u_1,\ldots,u_n}(p)}
< \frac{\sqrt{p} \big(\log(p)+2\log(2) \big)  M_0 p^{d-2}}{\frac{p^{d-1}}{(d-1)! u_1 u_2\cdots u_d} + O(p^{d-2})},$$ and so we have proved at
the same time the existence of good (non-bad) solutions, and that for large primes $p$, random weighted partitions are good with probability
tending to $1$ as $p$ approaches infinity.

To prove the general case $v>1$, we work using similar ideas, showing that $$|\BB| < \sqrt{p} \big(\log(p)+2\log(2) \big) M_0 \sum_{i=1}^v \big(
p^{d_i-2} \prod_{j\neq i} p^{d_j-1} \big),$$ where $M_0$ is a positive constant depending on $u(C_{j,i})$'s and $d$ (combinatorial constant as
above), but not on $p$. In addition, we know that the total number of positive integer solutions of $\S(\A)$ is $\prod_{i=1}^v \Big(
\frac{p^{d_i-1}}{(d_i-1)! u(C_{1,i}) u(C_{2,i})\cdots u(C_{d_i,i})} + O(p^{d_i-2}) \Big)$. Therefore, we conclude the same for $\S(\A)$ with
$v>1$.

Let $\{\mu_1,\ldots,\mu_d \}$ a good (non-bad) solution of $\S(\A)$. By Theorems \ref{ap4} and \ref{ap6} (in there, we take again $C=1$), we
have $$\Big|\sum_{i<j} s(p-\nu_i' \nu_j,p) D_i.D_j \Big|< \Big(\sum_{i<j} D_i.D_j \Big) (3\sqrt{p} + 5) \ \ \text{and} \ \ LCF< \Big(\sum_{i<j}
D_i.D_j \Big)(3\sqrt{p} + 2).$$ Moreover, by Example \ref{e41}, we have $$\big| CCF \big| < \Big(\sum_{i<j} D_i.D_j \Big)(6\sqrt{p} + 7).$$

Now, since there are good solutions for arbitrary large $p$, we obtain that the corresponding surfaces $X$ satisfy $c_1^2(X) \approx p
\bar{c}_1^2(Z,\A)$ and $c_2(X) \approx p \bar{c}_2(Z,\A)$ (see Propositions \ref{p42} and \ref{p43}). Therefore, if $\bar{c}_2(Z,\A)\neq 0$ and
$p$ approaches infinity, there are nonsingular projective surfaces $X$ having {\large $\frac{c_1^2(X) }{c_2(X)}$} arbitrarily close to {\large
$\frac{\bar{c}_1^2(Z,\A)}{\bar{c}_2(Z,\A)}$}.
\end{proof}

\begin{rmk} \textbf{Random partitions} of prime numbers are necessary in our construction, if we want to approach to the log Chern ratio of
the corresponding arrangement. If some of the numbers $p-\nu'_i \nu_j$ belong to the bad set $\FF$, then some of the summands in the error terms
of $c_1^2(X)$ and $c_2(X)$ may become $p$ proportional to some constant (see right after Theorem \ref{ap4} for an example), changing the
asymptotic limit of $\frac{c_1^2(X)}{c_2(X)}$. This distribution behavior is explained in \cite{Gi03} for the case of Dedekind sums.

In the following tables, we show samples of the actual Chern invariants of $X$, by means of a computer program \footnote{This program was
written in $C$++ (by the author) for the purpose of computing the exact Chern invariants of the surfaces $X$ coming from the data
$(Y,p,D,\L)$.}. For this example, we take $Z=\P^2(\C)$, and $\A=\text{CEVA}(3)$ (in Example \ref{e81}). The log Chern ratio associated to
$(Z,\A)$ is $\frac{8}{3}$. In the first table, we take $p= 61 \,169$ as a large prime number. For non-random-looking partitions of $61 \,169$,
we see that the Chern ratio of $X$ does not seem to be approaching to $\frac{8}{3}=2.\bar{6}$, in contrast with the random-looking ones. In the
second table, we see the asymptotic result in Theorem \ref{t71} by evaluating at several prime numbers.

\vspace{0.2cm}
\begin{center}
\begin{tabular}{|c|c|c|c|c|}
\hline
   \text{Partition of} $p=61\,169$     & $c_1^2(X)$ & $c_2(X)$ &  $\frac{c_1^2(X)}{c_2(X)}$ \\ \hline

\small{1+2+3+4+5+6+7+8+61$\,$133}                                                        & 1$\,$441$\,$949 & 733$\,$435 & 1.966\ldots \\ \hline

\small{1+29+89+269+1$\,$019+3$\,$469+7$\,$919+15$\,$859+32$\,$515}                       & 1$\,$465$\,$970 & 552$\,$166 & 2.654\ldots \\ \hline

\small{6$\,$790+6$\,$791+6$\,$792+6$\,$793+6$\,$794+6$\,$795+6$\,$796+6$\,$797+6$\,$821} & 1$\,$464$\,$209 & 633$\,$619 & 2.310\ldots \\ \hline

\small{1+100+300+600+1$\,$000+3$\,$000+8$\,$000+15$\,$000+33$\,$168}            & 1$\,$466$\,$250 & 561$\,$546 & 2.611\ldots \\ \hline

\small{1+30+90+270+1$\,$020+3$\,$470+7$\,$920+15$\,$860+32$\,$508}              & 1$\,$465$\,$778 & 553$\,$594 & 2.647\ldots \\ \hline

\small{1+32+94+276+1$\,$028+3$\,$474+7$\,$922+15$\,$868+32$\,$474}              & 1$\,$466$\,$575 & 552$\,$809 & 2.652\ldots \\ \hline

\small{1+1+1+1+1+1+1+1+61$\,$161}                                               & 1$\,$386$\,$413 & 1$\,$060$\,$303 & 1.307\ldots  \\ \hline

\small{1+1+89+89+1$\,$019+3$\,$469+7$\,$919+15$\,$859+32$\,$723}                & 1$\,$465$\,$370 & 553$\,$402 & 2.647\ldots \\ \hline

\small{1+23+45+100+1$\,$019+3$\,$002+16$\,$199+20$\,$389+20$\,$391}             & 1$\,$466$\,$285 & 573$\,$535 & 2.556\ldots \\ \hline

\end{tabular}
\vspace{0.3 cm}

Table for the dual Hesse arrangement and $p=61\,169$.
\end{center}
\vspace{0.1cm}

\begin{center}
\begin{tabular}{|c|c|c|}
\hline
   \text{Partition of} $p$                & $\frac{c_1^2(X)}{\chi(X,\O_X)}$ &  $\frac{c_1^2(X)}{c_2(X)}$ \\ \hline

\small{1+2+3+5+7+11+13+17+24=83}        & 7.331\ldots   & 1.570\ldots \\ \hline

\small{1+3+5+7+11+13+17+23+21=101}        &  7.503\ldots & 1.668\ldots \\ \hline

\small{1+3+7+13+19+23+47+67+59=239}       &  8.124\ldots & 2.096\ldots \\ \hline

\small{1+3+7+13+19+37+79+139+301=599}     &  8.390\ldots & 2.324\ldots \\ \hline

\small{1+3+7+17+29+47+109+239+567=1$\,$019}   &  8.408\ldots  & 2.341\ldots \\ \hline

\small{1+7+17+37+79+149+293+599+1$\,$087=2$\,$269}     &  8.586\ldots & 2.515\ldots \\ \hline

\small{1+11+23+53+101+207+569+1$\,$069+2$\,$045=4$\,$079}     &  8.646\ldots  & 2.578\ldots  \\ \hline

\small{1+23+53+101+207+449+859+1$\,$709+3$\,$617=7$\,$019}     &  8.685\ldots  & 2.620\ldots \\ \hline

\small{1+23+53+101+207+449+1$\,$709+2$\,$617+4$\,$943=10$\,$103}     & 8.695\ldots  & 2.631\ldots  \\ \hline

\small{1+29+89+269+1019+3$\,$469+7$\,$919+15$\,$859+32$\,$515=61$\,$169}     &  8.716\ldots  & 2.654\ldots \\ \hline

\small{1+101+207+569+1$\,$069+10$\,$037+22$\,$441+44$\,$729+66$\,$623=145$\,$777}     & 8.723\ldots & 2.662\ldots \\ \hline

\small{1+619+1$\,$249+2$\,$459+5$\,$009+10$\,$037+32$\,$323+68$\,$209+110$\,$421=230$\,$327}     & 8.725\ldots & 2.664\ldots \\ \hline

\small{1+929+1$\,$889+3$\,$769+6$\,$983+15$\,$013+32$\,$323+87$\,$443+163$\,$751=312$\,$101}     & 8.724\ldots & 2.663\ldots \\ \hline

\small{1+929+1$\,$889+3$\,$769+6$\,$983+15$\,$013+45$\,$259+90$\,$749+172$\,$397=336$\,$989}     & 8.725\ldots & 2.664\ldots \\ \hline

\small{1+929+1$\,$889+3$\,$769+6$\,$983+15$\,$013+45$\,$259+90$\,$749+187$\,$637=352$\,$229}     & 8.725\ldots & 2.664\ldots \\ \hline

\small{1+1$\,$709+3$\,$539+7$\,$639+15$\,$629+31$\,$649+62$\,$219+150$\,$559+271$\,$165=544$\,$109}     & 8.726\ldots & 2.665\ldots \\ \hline
\end{tabular}
\vspace{0.3 cm}

Table for the dual Hesse arrangement and various primes $p$.
\end{center}
\vspace{0.2cm}

For another example, take a general arrangement of $d$ lines (i.e. only nodes as singularities) in $\P^2(\K)$. Then, the corresponding log Chern
ratio tends to $2$ as $d$ approaches infinity. If we randomly choose partitions of $p$, we obtain surfaces $X$ whose Chern ratio tends to $2$.
If instead we choose all multiplicities equal to $1$ except by one (which is $p-d+1$), then we obtain surfaces with Chern ratio tending to $1.5$
as $d$ approaches infinity, whenever we take $d-1$ out of the bad set (we can do it because the bad set has measure over $p$ tending to $0$).
If, in the same case, we take $d=p$, then the limit Chern ratio is $1$. \label{r71}
\end{rmk}

The proof of Theorem \ref{t71} and the previous discussion lead us to the following definition.

\begin{defi}
Let $Z$ be a nonsingular projective surface over $\K$, and let $\A$ be a simple crossing divisible arrangement of curves in $Z$. Consider the
construction (in the proof of Theorem \ref{t71}) of surfaces $X$ coming from the pair $(Z,\A)$, for large prime numbers $p$ and good partitions
$\{\mu_i \}$ of $p$. We call any such $X$ a \underline{random surface} associated to $(Z,\A)$. \label{d71}
\end{defi}

In this way, Theorem \ref{t71} proves the existence of random surfaces associated to $(Z,\A)$ with Chern ratio arbitrarily close to the log
Chern ratio associated to $(Z,\A)$.

\begin{rmk}
If $\bar{c}_1^2(Z,\A)+\bar{c}_2(Z,\A)>0$ and $\bar{c}_1^2(Z,\A)>0$, then the random surfaces associated to $(Z,\A)$ are of general type. This is
because of Proposition \ref{p41}, and the Enriques' classification of surfaces (in Char$(\K)>0$, see \cite[p. 119-120]{Beauville96}).
\label{r72}
\end{rmk}

In general, random surfaces may not be minimal. However, because of numerical properties of negative-regular continued fractions, we conjecture
the following for their minimal models. Let us denote the minimal model of $X$ by $X_0$.

\begin{conj}
Let $Z$ be a nonsingular projective surface over $\K$, and let $\A$ be a simple crossing divisible arrangement of curves in $Z$. Assume that
$\bar{c}_1^2(Z,\A) >0$ and $\bar{c}_2(Z,\A)>0$. Then, there exist minimal nonsingular projective surfaces $X_0$ of general type with

\vspace{0.2 cm}
\begin{center} {\Large$\frac{c_1^2(X_0) }{c_2(X_0)}$} arbitrarily close to {\Large $\frac{\bar{c}_1^2(Z,\A)}{\bar{c}_2(Z,\A)}$}. \end{center}
\label{co71}
\end{conj}
\vspace{0.2 cm}

For the following two corollaries of Theorem \ref{t71}, we take $\K=\C$. The first corollary is a sort of ``uniformization" for algebraic
surfaces via Chern ratios. The second is the log Miyaoka-Yau inequality for simple crossing divisible arrangements \footnote{See Kobayashi
\cite{Kobayashi1} and Sakai \cite{Sakai1} for much more information about log Miyaoka-Yau inequalities.}.

\begin{cor}
Let $Z$ be a minimal nonsingular projective surface of general type over $\C$. Then, there exist nonsingular projective surfaces $X$, and
generically finite maps $f: X \rightarrow Z$ of high degree, such that
\begin{itemize}
\item[(i)] $X$ is minimal of general type. \item[(ii)] The Chern ratio $\frac{c_1^2(X)}{c_2(X)}$ is arbitrarily close to $2$. \item[(iii)]
$q(X)=q(Z)$.
\end{itemize}
\label{c71}
\end{cor}

\begin{proof}
Say $Z \hookrightarrow \P^m(\C)$ for some $m>2$. For integers $d\geq 4$, we consider simple normal crossing arrangements $\A=\{H_1,\ldots,H_d
\}$, where $H_i$ is a nonsingular hyperplane section of $Z$. This is a divisible arrangement. Since $Z$ is minimal of general type, we have that
$5 K_Z \sim C$, where $C$ is a nonsingular projective curve with $C^2>0$. This is because $|5 K_Z|$ defines a birational map into its image,
which is an isomorphism outside of finitely many ADE configurations of $(-2)$-curves. We take $C$ such that $\A \cup C$ has only nodes as
singularities. Let $p$ be a large prime number, and let $f:X \rightarrow Z$ be the $p$-th root cover producing random surfaces $X$ as in Theorem
\ref{t71}. Notice that, in this case, we are considering partitions $\mu_1+\ldots+\mu_d=p$.

As in Proposition \ref{p22}, we have the $\Q$-numerical equivalence $K_X \equiv f^*\big(K_Z + \frac{p-1}{p} \sum_{i=1}^d H_i \big) + \Delta$.
Assume there is a $(-1)$-curve $\Gamma$ in $X$. Then, $K_X.\Gamma=-1$. We know that $f^*(K_Z).\Gamma=f^*(\frac{1}{5}C).\Gamma \geq 0$. On the
other hand, as we pointed out in Remark \ref{r21}, we have that $f^*\Bigl( \frac{(p-1)}{p} \sum_{i=1}^d H_i \Bigr) + \Delta$ is an effective
$\Z$-divisor. Thus $\Gamma$ has to be a component of $f^*(D)$, where $D= \sum_{i=1}^d \mu_i H_i$. But all curves occurring on toric resolutions
have self-intersection $\leq -2$, and so, for some $i$, we have $\Gamma=\widetilde{H}_i$, where $\widetilde{H}_i$ is the strict transform of
$H_i$ under $f$. But, one can easily compute the self-intersection $$\widetilde{H}_i^2 = \frac{\deg(Z)}{p} \Big( 1 - \sum_{j\neq i} (p-\nu'_i
\mu_j) \Big).$$ Moreover, $\sum_{j \neq i} \nu'_i \nu_j \leq (d-2)p-1$, and $\deg(Z) \geq 3$ since $Z$ is not rational, and so $(d-2)p-1 \leq
\Big(d-1- \frac{2}{\deg(Z)} \Big)p -1$. By rearranging terms, we obtain $\widetilde{H}_i^2 \leq -2$, which is contradiction. Therefore, $X$ is
minimal.

By Proposition \ref{p51}, the log Chern numbers of $(Z,\A)$ are $\bar{c}_1^2(Z,\A) = \deg(Z) d^2 + \big( 4g(H_1)-4-2\deg(Z) \big)d +c_1^2(Z)$,
and $\bar{c}_2(Z,\A) = \frac{\deg(Z)}{2} d^2 + \big( 2g(H_1)-2-\frac{\deg(Z)}{2} \big)d + c_2(Z)$. We take $d$ large enough, so that
$\bar{c}_1^2(Z,\A)>0$ and $\bar{c}_2(Z,\A)>0$. Hence, the random surface $X$ is of general type. Moreover, the Chern ratio $\frac{c_1^2(X)
}{c_2(X)}$ is arbitrarily close to $\frac{\bar{c}_1^2(Z,\A)}{\bar{c}_2(Z,\A)}$. But when $d$ is large, this log Chern ratio tends to $2$.

Finally, we notice that any curve in $\A$ is very ample. In particular, as a direct application of the Viehweg vanishing Theorem in
\cite{ViVanishing}, we have $q(X)=q(Z)$.
\end{proof}

\begin{cor}
Let $Z$ be a nonsingular projective surface over $\C$, and let $\A$ be a simple crossing divisible arrangement of curves in $Z$. Assume that
$\bar{c}_1^2(Z,\A)>0$ and $\bar{c}_2(Z,\A) >0$. Then, $\bar{c}_1^2(Z,\A) \leq 3 \bar{c}_2(Z,\A)$. \label{c72}
\end{cor}

\begin{proof}
The inequalities $\bar{c}_1^2(Z,\A)>0$ and $\bar{c}_2(Z,\A) >0$ imply that the random surfaces $X$ associated to $(Z,\A)$ are of general type,
by Proposition \ref{p41} and Enriques' classification of surfaces. In this way, the Miyaoka-Yau inequality $c_1^2(X) \leq 3 c_2(X)$ holds. We
now use the limit procedure in the proof of Theorem \ref{t71} to find the corresponding log inequality.
\end{proof}

In general, log Miyaoka-Yau inequalities may be more restrictive when one considers a fixed surface $Z$ together with a fixed type of simple
crossing divisible arrangements, in the sense that $\bar{c}_1^2 \leq a \bar{c}_2$ for some number $a<3$. In addition, the constant $a$ may
depend on the ground field $\K$. In the next section, we will see an instance of this situation.

\section{Line arrangements in $\P^2(\K)$.} \label{s8}

Line arrangements on the plane form a very important class of simple crossing divisible arrangements. In this section, we show some key
examples, and we prove constrains for their log Chern invariants.

\begin{example}
Let $m\geq 1$ be an integer. We define the \underline{CEVA arrangement} of degree $m$ (see \cite[p. 435]{Do04}) through the zeros of the
equation $$(x^m-y^m)(y^m-z^m)(x^m-z^m)=0$$ in $\P^2(\C)$. It is denoted by CEVA$(m)$. We do not specify the order of its $3m$ lines. CEVA$(1)$
is a triangle. CEVA$(2)$ is the complete quadrilateral, having $d=6$, $t_2=3$, $t_3=4$, and $t_n=0$ else. CEVA$(3)$ is the \underline{dual Hesse
arrangement}, whose combinatorial data is given by $d=9$, $t_3=12$, and $t_n=0$ otherwise. For $m\geq 4$, CEVA$(m)$ has $d=3m$, $t_3=m^2$,
$t_m=3$ and $t_n=0$ for $n\neq 3,m$. These arrangements are rigid, in the sense that any other line arrangement with the same combinatorial data
is projectively equivalent to CEVA$(m)$ for some $m$ (see \cite{Urzua2}, where this is a particular case of a $(3,m)$-net corresponding to $\Z/
m \Z$). For $m \geq 2$, these arrangements provide examples with large log Chern ratio, equal to
$$\frac{\bar{c}_1^2(\P^2(\C),\text{CEVA}(m))}{\bar{c}_2(\P^2(\C),\text{CEVA}(m))} = \frac{5m^2-6m-3}{2m^2-3m}.$$ The highest value is attained
only by CEVA$(3)$, and is $\frac{8}{3}$. \label{e81}
\end{example}

The following theorem tells us that we cannot do better than $\frac{8}{3}$ with complex line arrangements, and that there is also a constrain in
positive characteristic. Example \ref{e82} shows that this constrain in sharp. Part $2.$ of Theorem \ref{t81} was proved by Sommese in
\cite[Theorem (5.3)]{Sommese1}, in the spirit of Hirzebruch's article \cite{HiLines83}, and part $3.$ is a known fact coming from the underlying
real topology of $\P^2(\R)$ (see for example \cite{IiLines78} or \cite{HiLines83}).

\begin{thm}
Let $\A$ be a line arrangement in $\P^2(\K)$.
\begin{itemize}
\item[1.] If Char$(\K)>0$, then $\bar{c}_1^2(\P^2(\K),\A) \leq 3 \bar{c}_2(\P^2(\K),\A)$. \item[2.] If Char$(\K)=0$, then $
\bar{c}_1^2(\P^2(\K),\A) \leq \frac{8}{3} \bar{c}_2(\P^2(\K),\A)$. Equality holds if and only if $\A$ is a triangle, or $t_{d-1}=1$, or the dual
Hesse arrangement. \item[3.] For arrangements defined over $\R$, we have $ \bar{c}_1^2(\P^2(\R),\A) \leq \frac{5}{2} \bar{c}_2(\P^2(\R),\A)$.
\end{itemize}
\label{t81}
\end{thm}

\begin{proof}
($1.$) First, we notice that $\bar{c}_1^2(\P^2(\K),\A) \leq 3 \bar{c}_2(\P^2(\K),\A)$ is equivalent to $\sum_{n\geq 2} t_n \geq d$, because of
Proposition \ref{p51}. Let $\sigma: \text{Bl}_{n-\text{pts}}(\P^2(\K)) \rightarrow \P^2(\K)$ be the blow-up of $\P^2(\K)$ at all the $n$-points
of $\A$ ($2$-points included). Then, $\Pic \big(\text{Bl}_{n-\text{pts}}(\P^2(\K)) \big) \otimes \Q$ has $\Q$-dimension $1+ \Sigma_{n\geq 2}
t_n$ \cite[Ch. V.3]{Hartshorne}. Assume $\sum_{n\geq 2} t_n < d$. Let $\{L_1,\ldots,L_d\}$ be the proper transforms under $\sigma$ of the lines
in $\A$, and let $H$ be the class of the pull-back of a general line. Since $t_d=0$, we have $L_i^2 \leq -1$ for all $i$. Also, for $i\neq j$,
$L_i.L_j=0$. Therefore, they are linearly independent in $\Pic \big( \text{Bl}_{n-\text{pts}}(\P^2(\K)) \big) \otimes \Q$, and since
$\sum_{n\geq 2} t_n < d$, they form a base (so $d=1+ \Sigma_{n\geq 2} t_n$). In this way, there exist $\alpha_i \in \Q$ such that $H=
\Sigma_{i=1}^d \alpha_i L_i$. But $L_i.H=1=\alpha_i L_i^2$ and $H.H=1= \Sigma_{i=1}^d \alpha_i$, and so $1= \Sigma_{i=1}^d \frac{1}{L_i^2}$.
However, $L_i^2 \leq -1$, which is a contradiction.

($2.$) By the Lefschetz's principle \cite{LefschetzPrin}, we can assume that $\K=\C$. A direct consequence of Proposition \ref{p51} and
Hirzebruch's inequality \cite[p. 140]{HiLines83} $$t_2 + \frac{3}{4}t_3 \geq d + \sum_{n>4} (n-4)t_n,$$ reduces the problem to prove $t_2 +
\frac{1}{4} t_3 \geq 3$. The latter inequality can be proved assuming the contrary, and checking the impossibility of the corresponding few
cases. For the equality, we arrive easily to the triangle and $t_{d-1}=1$ (trivial cases), and to the combinatorial case $d=9$, $t_3=12$, and
$t_n=0$ else. In this case, we write $\A=\{L_1,L_2,\ldots,L_9 \}$ such that $L_1\cap L_2 \cap L_3$ is one of the twelve $3$-points. Since over
any line of $\A$ there are exactly four $3$-points, there is a $3$-point outside of $L_1 \cup L_2 \cup L_3$. Say $L_4 \cap L_5 \cap L_6$ is this
point, then $L_7 \cap L_8 \cap L_9$ gives another $3$-point. This gives a $(3,3)$-net with three special members $\{ L_1,L_2,L_3 \}$, $\{
L_4,L_5,L_6 \}$ and $\{ L_7,L_8,L_9 \}$ (see \cite{Urzua2} for the definition of a net). One can prove that this $(3,3)$-net is unique up to
projective equivalence (see for example \cite{Urzua2}). This arrangement is projectively equivalent to the dual Hesse arrangement.

($3.$) This is a simple computation using the real topology of $\P^2(\R)$ through the fact that any line arrangement induces a cell
decomposition of $\P^2(\R)$.
\end{proof}

\begin{rmk}
Theorem \ref{t81} raises the question of whether the inequality $\bar{c}_1^2(\P^2(\C),\A) \leq \frac{8}{3} \bar{c}_2(\P^2(\C),\A)$ is a
topological consequence, in analogy to the case over $\R$. We remark that this log inequality over $\C$ relies on the Miyaoka-Yau inequality for
complex algebraic surfaces, and a result of Sakai (see \cite{HiLines83}). \label{r81}
\end{rmk}

\begin{example}
Let $\K$ be an algebraically closed field of Char$(\K)=m>0$. In $\P^2(\K)$, we have $m^2+m+1$ points with coordinates in $\F_m$, and there are
$m^2+m+1$ lines such that through each of these points passes exactly $m+1$ of these lines, and each of these lines contains exactly $m+1$ of
these points \cite[p. 426]{Do04}. These lines define an arrangement of $d=m^2+m+1$ lines, denoted by PG$(2,m)$, which we call
\underline{projective plane arrangement}. When $m=2$, this is the famous \underline{Fano arrangement}. Its combinatorial data is
$t_{m+1}=m^2+m+1$ and $t_n=0$ otherwise, and its log Chern numbers are $$\bar{c}_1^2(\P^2(\K),\text{PG}(2,m))= 3(m+1)(m-1)^2 \ \ \ \text{and} \
\ \ \bar{c}_2(\P^2(\K),\text{PG}(2,m))=(m+1)(m-1)^2,$$ and so $\bar{c}_1^2(\P^2(\K),\text{PG}(2,m))=3\bar{c}_2(\P^2(\K),\text{PG}(2,m))$ for
every $m$.

Hence, these arrangements provide examples for which equality holds in part $1.$ of Theorem \ref{t71}, for any characteristic. \label{e82}
\end{example}

\section{Simply connected surfaces with high Chern ratio.} \label{s9}

Throughout this and the next sections, we assume $\K=\C$. In this section, we will show how to compute the topological fundamental group of some
surfaces coming from $p$-th root covers. In particular, this allows us to find simply connected surfaces associated to certain pairs $(Z,\A)$.
This is an important characteristic of $p$-th root covers, which is not true for abelian coverings in general.

\begin{defi}
Let $S$ be a nonsingular projective surface, and let $B$ be a nonsingular projective curve. A \underline{fibration} is a surjective map $g: S
\rightarrow B$ with connected fibers. \label{d91}
\end{defi}

The following is a well-known fact about fibrations (see for example \cite{XiaoFundGr91}). We denote the topological fundamental group of $A$ by
$\pi_1(A)$.

\begin{prop}
Let $g: S \rightarrow B$ be a fibration. If $g$ has a section and a simply connected fiber, then $\pi_1(S) \simeq \pi_1(B)$. \label{p91}
\end{prop}

\begin{prop}
Let $Z$ be a nonsingular projective surface over $\C$, and let $\A$ be a simple crossing divisible arrangement of curves in $Z$. Let $(Y,\nA)$
be the associated pair. Assume there exists a fibration $g: Y \rightarrow B$ such that $\nA$ contains a section and a simply connected fiber of
$g$. Let $f: X \rightarrow Y$ be the $p$-th root cover with data $(Y,p,D,\L)$ in Section \ref{s6}.

Then, $\pi_1(X) \simeq \pi_1(B)$. \label{p92}
\end{prop}

\begin{proof}
Notice that we are assuming $p$ large enough, so that $\S(\A)$ has solutions. Also, the map $f: X \rightarrow Y$ is totally branch along $\nA$
(see Section \ref{s6}). Since $\nA$ contains a section of $g$, we have an induced fibration $\bar{g}: X \rightarrow B$, and $\bar{g}$ has a
section. Moreover, the inverse image under $f$ of the simply connected fiber of $g$ in $\nA$ is a simply connected fiber of $\bar{g}$, because
in the $p$-th root cover process, toric resolutions only add chains of nonsingular rational curves over the nodes of $\nA$. Therefore, by
Proposition \ref{p91}, the isomorphism $\pi_1(X) \simeq \pi_1(B)$ holds.
\end{proof}

\begin{example}
It is easy to find arrangements for which Proposition \ref{p92} applies. For example, consider the pencil $\PP: \ a(x^m-y^m) + b(y^m-z^m)=0$ in
$\P^2(\C)$ with $m>1$. The general fiber is a Fermat curve of degree $m$, and $\PP$ has exactly three singular fibers: $x^m-y^m$, $y^m-z^m$, and
$x^m-z^m$. They form the arrangement $\A=$ CEVA$(m)$ in Example \ref{e81}. Consider the associated pair $(Y,\nA)$. Then, there is a fibration
$g: Y \rightarrow \P^1(\C)$ with three simply connected singular fibers, and $m^2$ sections. Therefore, the corresponding surfaces $X$ coming
from $(\P{_\C}^2,\text{CEVA}(m))$ are simply connected. By the computations in Example \ref{e81}, we conclude that the random surfaces
associated to $(\P{_\C}^2,\text{CEVA}(m))$ are simply connected of general type with Chern ratio arbitrarily close to
$\frac{5m^2-6m-3}{2m^2-3m}$. \label{e91}
\end{example}

We can actually say more in the case of complex line arrangements.

\begin{thm} Let $\A$ be a line arrangement in $\P^2(\C)$. Assume that $\bar{c}_2(\P^2(\C),\A) \neq 0$. Then, there exist nonsingular simply connected
projective surfaces $X$ over $\C$ with \begin{center} {\large $\frac{c_1^2(X) }{c_2(X)}$} arbitrarily close to {\large
$\frac{\bar{c}_1^2(\P^2(\C),\A)}{\bar{c}_2(\P^2(\C),\A)}$} $\leq \frac{8}{3}$. \end{center} Moreover, {\large $\frac{c_1^2(X) }{c_2(X)}$} is
arbitrarily close to $\frac{8}{3}$ if and only if $\A$ is the dual Hesse arrangement. \label{t91}
\end{thm}

\begin{proof}
Let $\A=\{L_1,\ldots,L_d\}$ be an arbitrary arrangement of $d$ lines in $Z=\P^2(\C)$ (with $t_d=0$ as always). Let $P$ be a point in $L_1$ which
is nonsingular for $\A$. We consider the trivial pencil $\PP: \alpha L_1 + \beta L=0$, where $L$ is a fixed line in $\P^2(\C)$ passing through
$P$. Let $f: X \rightarrow Y$ be the $p$-th root cover associated to $(Z,\A)$ as in Theorem \ref{t71}. Let us denote the cover data by
$(Y,p,D,\L)$. Let $\tau: Y' \rightarrow Y$ be the blow-up of $Y$ at $P \in \nA$, and let $E$ be the exceptional curve over $P$. Consider the
divisor $D'= \tau^*(D)$. Then, we naturally have a $p$-th root cover $f': X' \rightarrow Y'$ with data $(Y',p,D',\tau^*(\L))$, and so a
birational map $\varsigma: X' \rightarrow X$, which is an isomorphism when restricted to $X'\setminus f'^{-1}(E)$. This map is actually a
birational morphism sending $f'^{-1}(E)$ to the point $f^{-1}(P)$. In particular, $\pi_1(X') \simeq \pi_1(X)$.

To compute $\pi_1(X')$, we look at the fibration $g:X \rightarrow \P^1(\C)$ induced by the pencil $\PP$. Notice that $E$ produces a section for
$g$. Moreover, if $L'_1$ is the strict transform of $L_1$ under $\tau$, then $f'^{-1}(L'_1)$ is simply connected, formed by a tree of
$\P^1(\C)$'s. Therefore, by Proposition \ref{p91}, $X'$ is simply connected. The rest of the Theorem follows from Theorem \ref{t71}, and part
$2.$ of Proposition \ref{t81}.
\end{proof}

\section{New record for Chern ratios of simply connected surfaces.} \label{s10}

Let $m\geq 4$ be an integer. In this section, we use random surfaces to construct examples of nonsingular simply connected projective surfaces
of general type, with Chern ratio arbitrarily close to $\frac{5m^2-12m+6}{2m^2-6m+6}$. Hence, for $m \in \{4,5,6 \}$, we have
$\frac{5m^2-12m+6}{2m^2-6m+6}>2.\overline{703}$, and so these surfaces improve the current record for Chern ratios given by Persson, Peters, and
Xiao in \cite{PerPetXiao96}. At the same time, these examples show that random surfaces may provide a tool to attack the open problem of finding
simply connected surfaces of general type with Chern ratio arbitrarily close to $3$. Our method encodes this problem into the existence of a
suitable arrangement of curves.

Let us consider the line arrangement CEVA$(m)$ in $\P^2(\C)$ (see Example \ref{e81}). It has three $m$-points. Let $\tau: Z \rightarrow
\P^2(\C)$ be the blow-up at these three points. Let \underline{CEVA}$(m)$ be the proper transform of CEVA$(m)$ under $\tau$. In this way, the
arrangement \underline{CEVA}$(m)$ does not contain any of the exceptional divisors of $\tau$, and so it is not ``log-equivalent" to CEVA$(m)$.

We observe that \underline{CEVA}$(m)$ is formed by three arrangements $\A_1$, $\A_2$, and $\A_3$. Each of them is formed by $m$ rational curves
coming from the proper transforms of the $m$ lines passing through each of the $m$-points of CEVA$(m)$. Moreover, for each $i$, we have that
$\O_{Z}(C) \simeq \O_{Z}(C')$ for all curves $C,C'$ in $\A_i$. Therefore, it is easy to verify that $$ \text{\underline{CEVA}}(m)= \A_1 \cup
\A_2 \cup \A_3
$$ is a simple crossing divisible arrangement in $Z$ (see Definitions \ref{d51} and \ref{d52}). Its combinatorial data is given by $d=3m$, $t_3=m^2$,
and $t_n=0$ for $n\neq m$, and all of its curves have self-intersection equal to zero. By Proposition \ref{p51}, the log Chern numbers
corresponding to $(Z,\text{\underline{CEVA}}(m))$ are $$\bar{c}_1^2(Z,\text{\underline{CEVA}}(m))= 5m^2-12m+6 \ \ \ \ \ \ \
\bar{c}_2(Z,\text{\underline{CEVA}}(m)) = 2m^2-6m+6.$$

By Theorem \ref{t71}, there exist random surfaces $X$ associated to $(Z,\text{\underline{CEVA}}(m))$ with $\frac{c_1^2(X)}{c_2(X)}$ arbitrarily
close to $\frac{5m^2-12m+6}{2m^2-6m+6}$. Moreover, by Remark \ref{r72}, they are of general type. We notice that the minimal models of the
surfaces $X$ may only improve this ratio, that is, make it closer to $3$. However, we conjecture that the asymptotic result remains unchanged
(see Conjecture \ref{co71}).

Each exceptional divisor of $\tau$ induces a fibration $Z \rightarrow \P^1(\C)$. By fixing any of them, we see that \underline{CEVA}$(m)$
contains a simply connected fiber and a section. We can now use Proposition \ref{p92} to prove that the random surfaces $X$ associated to
$(Z,\text{\underline{CEVA}}(m))$ are simply connected. The maximum value of $\frac{5m^2-12m+6}{2m^2-6m+6}$ in the range $m\geq 4$ is
$\frac{71}{26}$, which is attained by $m=5$. This proves the following theorem.

\begin{thm}
There exist nonsingular simply connected projective surfaces of general type with Chern ratio arbitrarily close to $\frac{71}{26}\approx
2.730769$. \label{t101}
\end{thm}

We finish with a concrete example satisfying $\frac{c_1^2}{c_2}> 2.\overline{703}$. It comes from $\text{\underline{CEVA}}(5)$. Consider the
following partitions of the prime number $61\,169$,

\vspace{0.2cm}
\begin{center} \begin{tabular}{c}
$1 + 307 + 7\,031 + 11\,109 + 42\,721 =61\,169$ \\
$589 + 2\,007 + 5\,007 + 20\,001 + 33\,565 = 61\,169$ \\
$1\,009 + 3\,001 + 13\,003 + 17\,807 + 26\,349 = 61\,169.$
\end{tabular}
\end{center}
\vspace{0.2cm}

As before, we split $\text{\underline{CEVA}}(5)$ into three subarrangements. Each of the equalities above assigns multiplicities to the curves
in each of these subarrangements. Then, via our procedure, we obtain a nonsingular simply connected projective surface $X$ of general type with
$c_1^2(X)=4\,341\,016$ and  $c_2(X)=1\,595\,264 $, and so $$\frac{c_1^2(X)}{c_2(X)} = \frac{542\,627}{199\,408} \approx 2.72119.$$

\section{Appendix: Dedekind sums and continued fractions.} \label{ap}

Most of the material in this appendix can be found in several places. Let $p$ be a prime number, and let $q$ be an integer satisfying $0<q<p$.

As before, we write the Dedekind sum associated to $(q,p)$ as $$s(q,p)=\sum_{i=1}^{p-1}
\Bigl(\Bigl(\frac{i}{p}\Bigr)\Bigr)\Bigl(\Bigl(\frac{iq}{p}\Bigr)\Bigr)$$ where $((x))=x-[x]-\frac{1}{2}$ for any rational number $x$. On the
other hand, we have the negative-regular continued fraction $$ \frac{p}{q} = e_1 - \frac{1}{e_2 - \frac{1}{\ddots - \frac{1}{e_s}}}$$ which we
abbreviate as $\frac{p}{q}=[e_1,...,e_s]$. According to our previous notation, $s=l(q,p)$. This continued fraction is defined by the following
recursion formula: let $b_{-1}=p$ and $b_0=q$, and define $e_i$ and $b_i$ by means of the equation $b_{i-2}=b_{i-1}e_i - b_i$ with $0\leq b_i <
b_{i-1}$ for $i \in \{1,2,\ldots,s \}$. In this way,
$$b_s=0<b_{s-1}=1<b_{s-2}<\ldots<b_1<b_0=q<b_{-1}=p.$$
In particular, $1\leq s \leq p-1$. By induction, one can prove that for every $i \in \{1,2,\ldots,s \}$, we have $b_{i-2}=(-1)^{s+1-i}
\det(M_i)$, where $M_i$ is the matrix
\[ \left(
\begin{array}{ccccccc}
-e_{i} & 1 & 0 & 0 &  \ldots & 0 \\
1 & -e_{i+1} & 1 & 0 & \ldots & 0 \\
0 & 1 & \ddots & \ddots & \ddots & \vdots \\
\vdots & \ddots & \ddots & \ddots & 1 & 0 \\
0 & \ldots & 0 & 1 & -e_{s-1} & 1 \\
0 & \ldots & 0 & 0 & 1 & -e_s \\
\end{array} \right)\]
Hence, $s= p-1$ if and only if $e_i=2$ for all $i$.

Another well-known way to look at this continued fraction is the following. Let us define the matrix
$$A(e_1,e_2,...,e_s) = \left( \begin{array}{cc} e_s & 1 \\ -1 & 0 \end{array} \right) \left( \begin{array}{cc} e_{s-1} & 1 \\ -1 & 0
\end{array} \right) \cdots \left( \begin{array}{cc} e_1 & 1 \\ -1 & 0
\end{array} \right),$$ and the recurrences $P_{-1}=0$, $P_0=1$, $P_{i+1}=e_{i+1} P_i - P_{i-1}$; $Q_{-1}=-1$, $Q_0=0$, $Q_{i+1}=e_{i+1}Q_i - Q_{i-1}$.
Then, again by induction, one can show that $\frac{P_i}{Q_i}=[e_1,e_2,...,e_i]$, and $$A(e_1,...,e_i)=\left( \begin{array}{cc} P_i & Q_i \\
-P_{i-1} & -Q_{i-1} \end{array} \right)$$ for all $i \in \{1,2,...,s \}$. The following lemma is proved using that det$\big(A(e_1,...,e_i)\big)
= 1$.

\begin{lem}
Let $p$ be a prime number and $q$ be an integer such that $0<q<p$. Let $q'$ be the integer satisfying $0<q'<p$ and $qq'\equiv 1 (mod \ p)$.
Then, $\frac{p}{q}=[e_1,...,e_s]$ implies $\frac{p}{q'}=[e_s,...,e_1]$. \label{ap1}
\end{lem}

We now express the numbers $\alpha_i= -1 + \frac{b_{i-1}}{p} + \frac{b'_{s-i}}{p}$ in terms of $P_i$'s and $Q_i$'s (these numbers appear in
Proposition \ref{p42}). Since
$$A(e_1,...,e_s)= \left( \begin{array}{cc} b_{i-1} & b_i \\ x & y \end{array} \right) \left(
\begin{array}{cc} e_i & 1 \\ -1 & 0 \end{array} \right) \left( \begin{array}{cc} P_{i-1} & Q_{i-1} \\ -P_{i-2} & -Q_{i-2} \end{array} \right)$$
we have $b_{i-1}=qP_{i-1}-pQ_{i-1}$ and $b'_{s-i}=P_{i-1}$.

\begin{lem}
$\sum_{i=1}^s \alpha_i(2-e_i) = \sum_{i=1}^s (e_i-2) + \frac{q+q'}{p} - 2 \frac{p-1}{p}$. \label{ap2}
\end{lem}

\begin{proof}
$\sum_{i=1}^s \alpha_i(2-e_i) = \sum_{i=1}^s (e_i-2) + \frac{1}{p} \sum_{i=1}^s \left( (q+1)P_{i-1} - p Q_{i-1} \right)(2-e_i)$. By definition,
$e_iP_{i-1}=P_i+P_{i-2}$ and $e_iQ_{i-1}=Q_i+ Q_{i-2}$, so $$\sum_{i=1}^s \left( (q+1)P_{i-1} - p Q_{i-1} \right)(2-e_i) = (q+1)\sum_{i=1}^s
(2P_{i-1}-P_i-P_{i-2})-p\sum_{i=1}^s (2Q_{i-1}-Q_i-Q_{i-2}).$$ $\sum_{i=1}^s (2P_{i-1}-P_i-P_{i-2})=1+P_{s-1}-p$ and $\sum_{i=1}^s
(2Q_{i-1}-Q_i-Q_{i-2})=-1+Q_{s-1}-q$, so $$ \sum_{i=1}^s \left( (q+1)P_{i-1} - p Q_{i-1} \right)(2-e_i) = q+P_{s-1}+2-2p$$ since
$qP_{s-1}-pQ_{s-1}=1$.
\end{proof}

Below, we describe the behavior of Dedekind sums and lengths of negative-regular continued fractions, for $p$ large and $q$ not in a certain bad
set. All of what follows relies on the work of Girstmair (see \cite{Gi03} and \cite{Gi06}).

\begin{defi}(from \cite{Gi03})
A \underline{Farey point} (F-point) is a rational number of the form $p \cdot \frac{c}{d}$, $1\leq d \leq \sqrt{p}$, $0\leq c \leq d$,
$(c,d)=1$. Fix an arbitrary constant $C>0$. The interval $$ I_{\frac{c}{d}} = \{ x: \ 0\leq x \leq p, \ \Big|x - p \cdot \frac{c}{d} \Big|\leq C
\frac{\sqrt{p}}{d^2} \} $$ is called the \underline{F-neighbourhood} of the point $p \cdot \frac{c}{d}$. We write $\FF_d = \bigcup_{c \in \CC}
I_{\frac{c}{d}}$ for the union of all neighbourhoods belonging to F-points of a fixed $d$, where $\CC= \{ c: \ 0\leq c \leq d \ \& \ (c,d)=1
\}$. The \underline{bad set $\FF$} is defined as $$\FF = \bigcup_{1\leq d \leq \sqrt{p}} \FF_d.$$ The integers $q$, $0\leq q < p$, lying in
$\FF$ are called \underline{F-neighbours}. Otherwise, $q$ is called an \underline{ordinary integer}. \label{ap3}
\end{defi}

The following two theorems are stated and proved in \cite{Gi03}.
\begin{thm}
Consider $p\geq 17$, and $q$ ordinary integer. Then, $|s(q,p)|\leq \big(2+\frac{1}{C} \big)\sqrt{p} + 5$. \label{ap4}
\end{thm}

The previous theorem is false for non-ordinary integers. For example, if $\frac{p+1}{m}$ is an integer, then $s(\frac{p+1}{m},p)= \frac{1}{12mp}
\big(p^2+(m^2-6m+2)p+m^2+1 \big)$. For more information see \cite{Gi03}.

\begin{thm}
For each $p\geq 17$ the number of F-neighbours is $\leq C\sqrt{p} \big(\log(p)+2\log(2) \big)$.  \label{ap5}
\end{thm}

A similar statement is true for lengths of negative-regular continued fractions.

\begin{thm}
Let $q$ be an ordinary integer, and let $\frac{p}{q}=[e_1,e_2,\ldots,e_s]$ be the corresponding continued fraction. Then, $s=l(q,p)\leq \big(
2+\frac{1}{C} \big)\sqrt{p} + 2$. \label{ap6}
\end{thm}

\begin{proof}
For any pair of integers $0<n<m$ with $(n,m)=1$, consider the regular continued fraction $$\frac{n}{m} = f_1 + \frac{1}{f_2 + \frac{1}{\ddots +
\frac{1}{f_r}}}.$$

Let us denote $\sum_{i=1}^r f_i$ by $t(n,m)$. Also, we write $\frac{n}{m}=[1,a_2,\ldots,a_{l'(n,m)}]$ for its negative-regular continued
fraction. Observe that $\frac{p}{q} - \big[ \frac{p}{q} \big]= \frac{x}{q}= [1,e_2,\ldots,e_s]$. By \cite[Corollary (iv)]{Myer87}, we have
$$t(q,p)=\Big[\frac{p}{q} \Big] + t(x,q)= \Big[\frac{p}{q} \Big] + l'(x,q) + l'(q-x,q)> l'(x,q)=s.$$ Now, by \cite[Prop. 3]{Gi06}, we
know that $t(q,p) \leq \big( 2+\frac{1}{C} \big)\sqrt{p} + 2$ whenever $q$ is an ordinary integer.
\end{proof}

\bibliographystyle{amsplain}


\vspace{0.3 cm} Department of Mathematics and Statistics, University of Massachusetts at Amherst, \\ 710 N. Pleasant Street, Amherst, MA 01003,
USA

{\small MSC class:  14J29}
\end{document}